\documentclass[11pt]{amsart}
\usepackage[T1]{fontenc}
\usepackage{lmodern}
\usepackage{amsmath,amssymb,amsthm,mathtools,mathrsfs}
\usepackage{booktabs}
\usepackage[margin=1in]{geometry}
\usepackage[hidelinks]{hyperref}

\newtheorem{theorem}{Theorem}[section]
\newtheorem{proposition}[theorem]{Proposition}
\newtheorem{lemma}[theorem]{Lemma}

\theoremstyle{remark}

\newtheorem{notation}[theorem]{Notation}

\newcommand{\C}{\mathcal C}
\newcommand{\D}{\mathcal D}
\newcommand{\ZC}{\mathcal Z(\mathcal C)}
\newcommand{\Irr}{\operatorname{Irr}}
\newcommand{\FPdim}{\operatorname{FPdim}}
\newcommand{\Tr}{\operatorname{Tr}}
\newcommand{\Gal}{\operatorname{Gal}}
\newcommand{\ord}{\operatorname{ord}}
\newcommand{\one}{\mathbf 1}
\newcommand{\Q}{\mathbb Q}
\newcommand{\Z}{\mathbb Z}
\newcommand{\CC}{\mathbb C}
\newcommand{\U}{\mathscr U}
\newcommand{\V}{\mathscr V}
\newcommand{\W}{\mathscr W}
\newcommand{\id}{\operatorname{id}}
\newcommand{\PSU}{\operatorname{PSU}}

\numberwithin{equation}{section}
\allowdisplaybreaks[1]

\newtheorem{example}[theorem]{Example}

\newcommand{\Rep}{\operatorname{Rep}}

\newcommand{\cC}{\mathcal C}
\newcommand{\cZ}{\mathcal Z(\mathcal C)}
\newcommand{\QQ}{\mathbb Q}
\newcommand{\ZZ}{\mathbb Z}
\newcommand{\FF}{\mathbb F}
\newcommand{\Vect}{\operatorname{Vec}}
\newcommand{\Hom}{\operatorname{Hom}}
\newcommand{\End}{\operatorname{End}}
\newcommand{\GL}{\operatorname{GL}}
\newcommand{\PGL}{\operatorname{PGL}}
\newcommand{\Ss}{\mathsf S}
\newcommand{\ii}{\mathrm i}

\newcommand{\ip}[2]{[#1,#2]}
\title[Grothendieck rings of near-Tambara--Yamagami categories]
{Grothendieck rings of near-Tambara--Yamagami fusion categories}
\author{Jingcheng Dong}
\address{College of Mathematics and Statistics, and Center for Applied Mathematics of Jiangsu Province, Nanjing University of Information Science and Technology, Nanjing 210044, China}
\email{jcdong@nuist.edu.cn}

\author{S\'ebastien Palcoux}
\address{Beijing Institute of Mathematical Sciences and Applications, Huairou District, Beijing 101408, China}
\email{sebastienpalcoux@gmail.com}

\author{Kai Wang}
\address{College of Mathematical Science, Yangzhou University, Yangzhou, China}
\email{wangkww@163.com}

\date{}
\begin{document}
\begin{abstract}
We determine the Grothendieck rings of fusion categories over $\CC$
containing a nonpointed Tambara--Yamagami fusion subcategory and exactly
one further simple object. The subcategory necessarily has rank five,
its invertible objects form the Klein four-group, and the coefficient
$\kappa$ of the additional simple object in its tensor square belongs
to $\{0,2,7\}$. Each of these three based rings is categorifiable.
For $\kappa=0$, every category belongs, up to tensor equivalence, to the
rank-one odd-component extension list of Choi--Lu--Sun. For $\kappa=2$
we construct an explicit unitary group-theoretical realization, and
for $\kappa=7$ a realization is $\Rep(\PSU(3,2))$.
\end{abstract}
\maketitle

\tableofcontents

\section{Introduction}
The classification of fusion categories is a structural problem at the meeting
point of representation theory, quantum topology, conformal field theory,
subfactor theory, and the mathematical theory of anyons. The Grothendieck ring
makes the fusion rules discrete and computable, but it does not determine an
associator: the pentagon equations impose additional, highly restrictive
constraints. Consequently, a fusion ring may satisfy every evident numerical
condition and still admit no categorification. Exact analyses of small and
structured families are therefore valuable both for the classification program and for producing reliable algebraic models for topological and physical constructions.

Tambara--Yamagami categories are canonical nonpointed examples in which all but
one simple object are invertible and the remaining object has Frobenius--Perron
dimension \(\sqrt{|G|}\) \cite{TY}. They are explicit enough for detailed
calculations, while already displaying irrational dimensions, nontrivial
associativity, and arithmetic constraints. Near-integral fusion categories,
formed by adjoining one simple object to a proper maximal-rank fusion
subcategory \cite{DS}, provide a natural next layer of complexity. The
near-Tambara--Yamagami case is especially well suited to a complete study: the Tambara--Yamagami subcategory supplies strong symmetry, whereas the extra simple object creates a sharp test of categorifiability. 

Following \cite{DS}, we write $\C=\C(\D,\kappa)$ for such a
near-Tambara--Yamagami category, where $\D$ is nonpointed. We denote by $G$ the group of invertible objects of $\D$, by $t$ its unique
noninvertible simple object, and by $\rho$ the unique simple object of
$\C\setminus\D$. The fusion rules are
\[
 |G|=n^2,\qquad \FPdim(t)=n,\qquad
 \rho\otimes \rho\cong \bigoplus_{g\in G}g\oplus n t\oplus\kappa \rho,
\]
so that, for $d=\FPdim(\rho)$,
\[
 d^2=\kappa d+2n^2.
\]

Our first result reduces the numerical possibilities to rank six. The
following table records the candidates produced by the formal-codegree and
indicator arguments; the irrational candidate $\kappa=8$ is ruled out
later by Theorem~\ref{r8:nonexistence}.

\begin{theorem}\label{thm:overview}
Let $\C=\C(\D,\kappa)$ be a near-Tambara--Yamagami fusion category over
$\CC$, with $\D$ nonpointed. The numerical analysis gives
\[
 n=2,\qquad G\cong C_2\times C_2,\qquad \operatorname{rank}(\C)=6,
\]
and the following candidate values:
\[
\begin{array}{c|c|c}
\kappa&d&\FPdim(\C)\\\hline
0&2\sqrt2&16\\
8&4+2\sqrt6&48+16\sqrt6\\
2&4&24\\
7&8&72
\end{array}
\]
The first two rows arise in the irrational case and the last two in the
integral case. The candidate $\kappa=8$ is impossible by
Theorem~\ref{r8:nonexistence}.
\end{theorem}

The irrational analysis is carried out in Section~\ref{sec:irrational}.
Formal codegrees first force $n=2$ and $\kappa$ to be a multiple of eight.
For positive multiples of eight, a Galois reflection on the Drinfeld center,
together with two Frobenius--Schur indicator traces, produces simultaneous
constraints on sums of roots of unity. A quadratic trace bound leaves only
$\kappa=8$ after the zero parameter has been separated. The subsequent
cyclotomic moment argument excludes that remaining candidate. The zero-parameter categories are covered by the rank-one
odd-component extension list of Choi--Lu--Sun \cite{CLS}.

Section~\ref{sec:rational} treats the integral case independently. An
integral parametrization gives two branches. The first branch leads to a
Tannakian subcategory and a finite group with a normal subgroup whose
nonidentity elements are conjugate; the second leads to a sharply doubly
transitive group. The first branch uses the braided Tambara--Yamagami restriction
and elementary arithmetic; the second uses the order bound on $G$.
Both force $n=2$, leaving $\kappa=7$ and $\kappa=2$, respectively.

The numerical result does not by itself classify associators. We therefore
state the categorical conclusions separately. The parameter $7$ is realized
by $\operatorname{Rep}(\PSU(3,2))$. For $\kappa=2$, Section~\ref{r2:sec:positive}
constructs an explicit unitary group-theoretical fusion category with the
fusion ring $R_2$.

For $\kappa\in\Z_{\geq0}$, let $R_\kappa$ be the based ring with
basis $G\cup\{t,\rho\}$, where $G=C_2\times C_2$, group multiplication
on $G$, and
\begin{equation}\label{eq:Rkappa}
\begin{gathered}
gt=tg=t,\qquad g\rho=\rho g=\rho\quad(g\in G),\qquad
t^2=\sum_{g\in G}g,\\
t\rho=\rho t=2\rho,\qquad
\rho^2=\sum_{g\in G}g+2t+\kappa\rho.
\end{gathered}
\end{equation}
The following theorem classifies the categorifiable based rings in the
family and records the available categorical conclusions.
\begin{theorem}\label{thm:classification}
Let $\C$ be a fusion category over $\CC$ containing a nonpointed
Tambara--Yamagami fusion subcategory $\D$ and exactly one simple object
outside $\D$. Then
\[
\operatorname{rank}(\D)=5,\qquad G\cong C_2\times C_2,\qquad
K_0(\C)\cong R_\kappa\quad\text{for }\kappa\in\{0,2,7\}.
\]
Each of these three based rings is categorifiable. More precisely:
\begin{enumerate}
\item For $\kappa=0$, every such category is tensor equivalent to an
entry in the eight-solution list of rank-one odd-component extensions
of $\Rep(H_8)$ in \cite[Section~7 and Appendix~C]{CLS}; see
Theorem~\ref{thm:kappa-zero} for the precise scope of this assertion.
\item For $\kappa=2$, Theorem~\ref{r2:prop:R2} provides an explicit
unitary group-theoretical categorification.
\item For $\kappa=7$, $\Rep(\PSU(3,2))$ is a categorification, as in
Example~\ref{example:PSU}.
\end{enumerate}
The dimensions are
\[
\begin{array}{c|c|c|c}
\kappa&\FPdim(t)&\FPdim(\rho)&\FPdim(\C)\\\hline
0&2&2\sqrt2&16\\
2&2&4&24\\
7&2&8&72.
\end{array}
\]
For $\kappa=2,7$, the assertion concerns based-ring isomorphism and
existence, not classification of all realizations up to tensor equivalence.
\end{theorem}

\section{Preliminaries}\label{sec:preliminaries}
\subsection{Pseudo-unitary fusion categories}
Throughout, let \(\mathcal C\) be a fusion category over the field of complex numbers $\mathbb{C}$.  We choose representatives
\[
    \Irr(\mathcal C)=\{X_1=\mathbf 1,X_2,\ldots,X_n\}
\]
for the isomorphism classes of simple objects, with \(\mathbf 1\) denoting the tensor unit.  They form the canonical \(\mathbb Z\)-basis of the Grothendieck ring \(K_0(\mathcal C)\), and all structure constants relative to this basis are nonnegative integers.

For simple objects \(X,Y\), write the tensor-product decomposition in the
form
\[
    X\otimes Y\cong
    \bigoplus_{Z\in\Irr(\mathcal C)}N_{XY}^{Z} Z.
\]
The nonnegative integer \(N_{XY}^{Z}\) is the fusion coefficient, namely the
multiplicity of \(Z\) in \(X\otimes Y\).  The multiplicity of \(\mathcal C\)
is the maximum of these integers as \(X,Y,Z\) range over \(\Irr(\mathcal C)\).
Rigidity yields the Frobenius reciprocity relations
\[
    N_{XY}^{Z}
    =N_{Z\,Y^*}^{X}
    =N_{X^*\,Z}^{Y}
    =N_{Y^*\,X^*}^{Z^*},
    \qquad X,Y,Z\in\Irr(\mathcal C).
\]

For each \(X\in\Irr(\mathcal C)\), let \(N_X\) be the fusion matrix whose
entries are given by
\[
    (N_X)_{Y}^{\,Z}=N_{XY}^{Z},
    \qquad Y,Z\in\Irr(\mathcal C).
\]
Its Perron--Frobenius eigenvalue is positive and is, by definition, the
Frobenius--Perron dimension of \(X\), denoted by \(\FPdim(X)\).  The
Frobenius--Perron theorem, together with \cite[Theorem 8.6]{etingof2005fusion},
shows that these numbers extend uniquely to a ring homomorphism
$\FPdim:K_0(\mathcal C)\longrightarrow\mathbb R$. The Frobenius--Perron dimension of $\C$ is
\[
    \FPdim(\mathcal C)
    =\sum_{X\in\Irr(\mathcal C)}\FPdim(X)^2.
\]

A pivotal structure on \(\mathcal C\) is a monoidal natural isomorphism
$i:\operatorname{Id}_{\mathcal C}\Longrightarrow(-)^{**}$,
where $(-)^*$ denotes the duality functor. Its pivotal trace defines
$\dim(X)$. The intrinsic global dimension is
\[
\dim(\C)=\sum_{X\in\Irr(\C)}\dim(X)\dim(X^*)
\]
by \cite[Proposition~2.9]{etingof2005fusion}. A pivotal structure is
spherical when $\dim(X)=\dim(X^*)$ for every simple $X$; only then may
this formula be written as $\sum_X\dim(X)^2$.
A fusion category is \emph{pseudo-unitary} if its intrinsic global
dimension equals its Frobenius--Perron dimension. It then has a unique
spherical structure with $\dim(X)=\FPdim(X)>0$ for every simple $X$
\cite[Proposition~8.23]{etingof2005fusion}. We use this positive
spherical structure whenever working with a pseudo-unitary category.
Unitarity is a stronger property and will be established separately
where it is needed.

\subsection{The Drinfeld center}

We will also use the Drinfeld center \(\mathcal Z(\mathcal C)\) of a fusion category $\C$.  Its objects are
pairs \((Z,c_{Z,-})\), where \(Z\) is an object of \(\mathcal C\) and $    c_{Z,X}:Z\otimes X\to X\otimes Z$
is a half-braiding, natural in \(X\in\mathcal C\) and satisfying the usual
compatibility with tensor products.  The center is a braided fusion category; in the standard spherical situation it is modular.

A principal tool that we will use is induction-restriction to the Drinfeld center. There is a forgetful functor $F:\ZC\to\mathcal C$, which sends \((Z,c_{Z,-})\) to its underlying object \(Z\).  This functor has a right adjoint $I:\mathcal C\to\ZC$,
usually called the induction functor.  Thus, for every \(X\in\mathcal C\) and \(Y\in\ZC\), there is a natural isomorphism  $\operatorname{Hom}_{\ZC}(Y,I(X)) \cong \operatorname{Hom}_{\mathcal C}(F(Y),X)$.

The underlying object of an induced object admits the standard description
\cite[Proposition 5.4]{etingof2005fusion}:
\begin{equation}\label{eq:induction-formula}
    F(I(X))\cong
    \bigoplus_{V\in\Irr(\mathcal C)}V\otimes X\otimes V^*.
\end{equation}
Because the forgetful functor preserves Frobenius--Perron dimensions, this
formula gives
\begin{equation}\label{dimI(X)}
\begin{aligned}
    \FPdim(I(X))=\FPdim(X)\,\FPdim(\mathcal C).
\end{aligned}
\end{equation}

Throughout, we write $[A,B]=\dim_{\CC}\operatorname{Hom}(A,B)$ for $A,B$ in a fusion category $\C$. Thus $[A,B]$ is exactly the multiplicity of $B$ in $A$ when $B$ is simple. Further, if we denote $[I(X),Z]$ by $m_X(Z)$, we get
\begin{equation}\label{eq:induction}
\begin{split}
 [Y,I(X)]&=[F(Y),X],\\
[I(X),I(Y)]&=\sum_Zm_X(Z)m_Y(Z).
\end{split}
\end{equation}

A  twist of a braided category $\C$ is a natural automorphism $\theta:\id_{\C}\to \id_{\C}$ satisfying $\theta_1=\id_1$ and $\theta_{X\otimes Y}=(\theta_X\otimes \theta_Y)c_{Y,X}c_{X,Y}$. Let $X$ be a simple object of $\C$. Then $\theta_X$ is a scalar in $\mathbb{C}$. It is well-known that any modular category is equipped with a twist $\theta$ such that $\theta_X$ is a root of unity for any $X\in \Irr(\C)$ \cite{Vafa1988421}.

When using spherical traces in the center, put $\Delta=\dim(\C)$;
in the positive spherical structure, $\Delta=\FPdim(\C)$.
For the canonical twist of the Drinfeld center $\ZC$,
\begin{equation}\label{eq:indicators}
\theta_{I(\one)}=\operatorname{id},\qquad
\Tr(\theta_{I(X)}^j)=\nu_j(X)\Delta\quad(j\geq1),
\end{equation}
where $\nu_j(X)$ is the $j$th Frobenius--Schur indicator of $X$.
These facts follow from \cite[Proposition~2.4 and Theorem~2.5]{Ostrik}
and \cite[Theorem~4.1]{NS}. For simple $X\ne \one$, $\nu_1(X)=0$;
$\nu_2(X)$ is $0$ if $X$ is not self-dual, and is $\pm1$ otherwise.
All simple twists in the Drinfeld center $\ZC$ are roots of unity.

We consistently write $\Ss$ for the unnormalized $S$-matrix of
$\ZC$, so $\Ss_{\one,Z}=\dim Z$, and $S=\Ss/\Delta$ for the
normalized matrix. The balancing identity is
\begin{equation}\label{eq:balancing-general}
\Ss_{V,W}=\frac{1}{\theta_V\theta_W}
\sum_Z[V\otimes W,Z]\theta_Z\dim Z.
\end{equation}
When $K_0(\C)$ is commutative, each irreducible character $\chi$
corresponds to a multiplicity-one simple summand $Z_\chi$ of $I(\one)$.
Writing $f_\chi=\sum_{X\in\Irr(\C)}\chi(X)\chi(X^*)$, one has
\begin{equation}\label{eq:central-character}
\dim Z_\chi=\frac{\Delta}{f_\chi},\qquad
\theta_{Z_\chi}=1,\qquad
\Ss_{Z_\chi,W}=\dim(Z_\chi)\chi(F(W)),\qquad
S_{Z_\chi,W}=\frac{\chi(F(W))}{f_\chi}.
\end{equation}
These are the formal-codegree correspondence and the pullback character
identity in \cite[Section~2.3, Example~2.9, and Theorem~2.13]{Ostrik}.
The normalization in \eqref{eq:central-character} will also be used for
the Galois action on $S$.

\subsection{Near-Tambara--Yamagami categories}\label{general-NTYC}
Inspired by near-group fusion categories, the authors \cite{DS} refer to fusion categories $\mathcal{C}=\mathcal{C}(\mathcal{D},\kappa)$ that admit a proper fusion subcategory $\mathcal{D}$ of maximal rank as \emph{near-integral}. In this paper, we consider the case that $\D$ is a Tambara--Yamagami category. In this case, we call  $\C$ a near-Tambara--Yamagami category.

\begin{example} \label{example:PSU}
The category $\Rep(\PSU(3,2))$ is a near-Tambara--Yamagami
category; see \cite[Example~3.7 and Figure~4]{DS} for its character table.
Write $G_7=\PSU(3,2)$. The irreducible complex representations of $G_7$ are denoted by
\[
\mathbf{1},\quad \lambda_1,\quad \lambda_2,\quad \lambda_3,\quad
t,\quad \rho,
\]
with dimensions
\[
\dim(\mathbf{1})=\dim(\lambda_i)=1,\qquad
\dim(t)=2,\qquad
\dim(\rho)=8.
\]

The fusion rules of $\Rep(G_7)$ are given by
\begin{align}
&\lambda_i\otimes\lambda_j
\cong (\lambda_i\lambda_j), \lambda_i\otimes t\cong t,
t\otimes t\cong
\mathbf{1}\oplus\lambda_1\oplus\lambda_2\oplus\lambda_3, \\
\lambda_i\otimes \rho
&\cong \rho, t\otimes \rho\cong 2 \rho, \rho\otimes \rho
\cong \mathbf{1}\oplus\lambda_1\oplus\lambda_2\oplus\lambda_3
\oplus 2 t\oplus 7 \rho.
\end{align}
Here $\lambda_i\lambda_j$ denotes the product in the character group,
and $i,j\in\{1,2,3\}$. Moreover,
\[
\lambda_1\otimes \lambda_1\cong \lambda_2\otimes \lambda_2\cong \lambda_3\otimes \lambda_3\cong \mathbf{1},
\quad
\lambda_1\otimes\lambda_2\cong \lambda_3.
\]
\end{example}

Throughout this subsection, $\mathcal{C}=\mathcal{C}(\mathcal{D},\kappa)$ is a near-Tambara--Yamagami category. Let $\rho$ denote the unique simple object not in $\mathcal{D}$.
Frobenius reciprocity implies
$X\otimes\rho\cong \rho\otimes X\cong \FPdim(X) \rho$ for every simple $X\in\D$; in particular,
$\D$ is integral. By \cite{TY}, the group $G$ consisting of invertible simple objects of $\D$ is abelian and the fusion rules are
\begin{equation}\label{eq:fusion}
\begin{gathered}
g\otimes t\cong t\otimes g\cong t,\quad g\otimes\rho\cong \rho\otimes g\cong \rho\quad(g\in G),\qquad |G|=n^2,\\
t\otimes t\cong \bigoplus_{g\in G}g,\qquad t\otimes\rho\cong \rho\otimes t\cong n \rho,\qquad
\rho\otimes \rho\cong \bigoplus_{g\in G}g\oplus n t\oplus\kappa \rho.
\end{gathered}
\end{equation}
Here $t$ is the unique noninvertible simple object of $\D$, $n=\FPdim(t)\in\Z_{\geq2}$, $\kappa\in\Z_{\geq0}$, and both
$t$ and $\rho$ are self-dual. The parameter $\kappa$ is the coefficient
of $\rho$ in $\rho^2$; it need not be the maximum fusion coefficient.
That maximum is $\max\{n,\kappa\}$, and is $\max\{2,\kappa\}$
after the rank reduction. In particular, both $R_0$ and $R_2$ have
maximum fusion coefficient two. Put
\begin{equation}\label{eq:dimensions-general}
N=2n^2,\qquad d=d_+=\frac{\kappa+\sqrt{\kappa^2+8n^2}}2,
\qquad d_-=\frac{\kappa-\sqrt{\kappa^2+8n^2}}2.
\end{equation}
Then $dd_-=-N$, $d^2=\kappa d+N$, and
$\Delta:=\FPdim(\C)=N+d^2=4n^2+\kappa d$.
Since $d$ is an algebraic integer, $d\notin\Z$ is equivalent to
$d\notin\Q$. We use ``irrational'' and ``rational'' in this sense;
the rational case is precisely the integral case for $\C$.

By \cite[Lemma~4.1]{DS}, a near-integral category is Galois conjugate to
a pseudo-unitary category. Galois conjugation preserves fusion rules.
We therefore use a pseudo-unitary conjugate whenever positivity of
categorical dimensions is needed. Integral categories are already
pseudo-unitary by \cite[Proposition~8.24]{etingof2005fusion}, and all pseudo-unitary categories below carry their
canonical positive spherical structures.

By (\ref{eq:induction-formula}), we have
\begin{equation}\label{eq:FI-general}
\begin{split}
F(I(g))&\cong n^2 g\oplus\bigoplus_{h\in G}2 h\oplus n t\oplus\kappa \rho,\\
F(I(t))&\cong \bigoplus_{h\in G}n h\oplus 3n^2 t\oplus n\kappa \rho,\\
F(I(\rho))&\cong \bigoplus_{h\in G}\kappa h\oplus n\kappa t\oplus (4n^2+\kappa^2) \rho.
\end{split}
\end{equation}
For a tensor functor, its image means the full fusion subcategory
generated by the simple constituents of its values. A tensor functor
$L:\mathcal A\to\mathcal B$ is dominant if every object of
$\mathcal B$ is a subobject of $L(X)$ for some $X\in\mathcal A$.
A tensor functor is dominant onto its image.

\begin{lemma}\label{lem:codegrees}
The characters of $K_0(\C)$ and their formal codegrees are
\[
\begin{array}{c|ccc|c}
&g\in G&t&\rho&f\\\hline
\chi_+&1&n&d&N+d^2\\
\chi_-&1&n&d_-&N+d_-^2\\
\eta&1&-n&0&N\\
\eta_\psi&\psi(g)&0&0&n^2
\end{array}
\]
where $\psi$ ranges over the nontrivial characters of $G$.
The object $I(\one)$ is multiplicity-free. Its summands corresponding to
these characters have dimensions
\begin{equation}\label{eq:unit-dimensions}
1,\qquad \frac{d^2}{N}=1+\frac\kappa N d,\qquad
2+\frac\kappa N d,\qquad 4+\frac{2\kappa}N d,
\end{equation}
respectively, the last with $n^2-1$ summands. Each forgetful image
contains $\one$ exactly once.
\end{lemma}
\begin{proof}
If a character is nontrivial on $G$, the relations $g\otimes t\cong t$ and
$g\otimes\rho\cong \rho$ force its values on $t,\rho$ to vanish. Otherwise
its value on $t$ is $\pm n$; the value $-n$ forces the value on
$\rho$ to be zero, while $n$ gives the two roots $d,d_-$.
This lists all $n^2+2$ characters. Their codegrees are
$f_\chi=\sum_X\chi(X)\chi(X^*)$. The formal-codegree theorem
\cite[Theorem~2.13]{Ostrik} assigns a summand of dimension
$\Delta/f_\chi$ to each character, with multiplicity one.
The displayed dimensions follow from $dd_-=-N$; the last statement
is adjunction.
\end{proof}

\begin{lemma}\label{lem:order-bound}
Every $g\in G$ satisfies
$\ord(g)\leq n$. In particular, if $n=2$ then $G\cong C_2\times C_2$.
\end{lemma}
\begin{proof}
The rank-one module category $\operatorname{add}(\rho)$ is equivalent,
as an abelian category, to $\operatorname{Vec}$. Under this equivalence,
the action of $X\in\D$ is tensoring with
$\Phi(X)=\operatorname{Hom}_{\C}(\rho,X\otimes\rho)$, and the module
associativity maps give natural isomorphisms
\[
 \Phi(X)\otimes\Phi(Y)\xrightarrow{\ \sim\ }\Phi(X\otimes Y),
 \qquad \CC\xrightarrow{\ \sim\ }\Phi(\one).
\]
On elementary tensors the first map sends $f\otimes g$ to
$(\id_X\otimes g)\circ f$. Thus $\Phi$ is strong monoidal. In
particular, it preserves duals, so $\Phi(t^*)\cong\Phi(t)^*$.

For the order estimate, let
\[
\Phi:\D\longrightarrow\operatorname{Vec},\qquad
\Phi(X)=\operatorname{Hom}_{\C}(\rho,X\otimes\rho),
\]
denote the functor just described. We use only its values and tensor
structure in the argument below.

For $g\in G$,
$g\otimes\rho\cong\rho$, so
\[
 \Phi(g)=\operatorname{Hom}_{\C}(\rho,g\otimes\rho)
 \cong\operatorname{End}_{\C}(\rho)\cong\CC.
\]
The first isomorphism uses any chosen isomorphism $g\otimes\rho\cong\rho$;
the last is Schur's lemma. Hence $\Phi(g)$ is one-dimensional. Likewise,
with $V=\Phi(t)$,
\[
 V\cong\operatorname{Hom}_{\C}(\rho,n \rho)\cong\CC^n,
 \qquad\dim_{\CC}V=n.
\]

Choose a nonzero vector $v_g\in\Phi(g)$; this is a basis of that line.
Choose an isomorphism $\psi_g:g\otimes t\to t$, which exists because
$g\otimes t\cong t$. The induced map
$\Phi(g)\otimes V\to\Phi(g\otimes t)$ followed by $\Phi(\psi_g)$
defines an invertible operator $U_g\in\operatorname{End}_{\CC}(V)$. The
choice of $\psi_g$ is unique up to a nonzero scalar because $t$ is simple.
Normalize $v_\one$ and $\psi_\one$ so that $U_\one=\operatorname{id}_V$.

For $g,h\in G$, both the composite of $v_g$ and $v_h$ and the chosen vector
$v_{gh}$ are nonzero elements of the one-dimensional space $\Phi(gh)$.
Thus their ratio is a scalar $c(g,h)\in\CC^\times$. Similarly, the two maps
from $(g\otimes h)\otimes t$ to $t$ obtained from $\psi_g,\psi_h$ and
from $\psi_{gh}$ differ by a nonzero scalar, since the relevant Hom space
is one-dimensional. Applying $\Phi$ and using monoidal coherence gives
\[
 U_gU_h=\alpha(g,h)U_{gh},\qquad \alpha(g,h)\in\CC^\times.
\]
Associativity implies the cocycle identity for $\alpha$, so the $U_g$
form a projective representation. Only the nonzero-scalar relation is
needed below.

Rigidity gives
$\operatorname{Hom}_{\D}(g\otimes t,t)\cong
\operatorname{Hom}_{\D}(g,t\otimes t^*)$. Since
$t\otimes t^*\cong\bigoplus_{h\in G}h$ with each summand once, the map
corresponding to $\psi_g$ is a nonzero multiple of the inclusion of the
$g$-summand. Therefore
$\bigoplus_{g\in G}g\to t\otimes t^*$ is an isomorphism.
Applying the strong monoidal functor $\Phi$, and using
$\Phi(t^*)\cong V^*$, gives
\[
 \bigoplus_{g\in G}\Phi(g)\cong V\otimes V^*\cong\operatorname{End}_{\CC}(V).
\]
Under the standard identification $V\otimes V^*\cong\operatorname{End}(V)$,
the basis vector $v_g$ maps to a nonzero scalar multiple of $U_g$.
Consequently
$\operatorname{End}_{\CC}(V)=\bigoplus_{g\in G}\CC U_g$, so the operators
$U_g$ are linearly independent.

If $\ell=\operatorname{ord}(g)$, repeated projective multiplication gives
$U_g^j=c_jU_{g^j}$ with $c_j\in\CC^\times$. The operators
$\operatorname{id}_V,U_g,\ldots,U_g^{\ell-1}$ are therefore linearly
independent. Cayley--Hamilton gives a nontrivial relation among the first
$n+1$ powers of the $n\times n$ matrix $U_g$. Hence $\ell\leq n$.
When $n=2$, $|G|=4$ and every nonidentity element has order at most two.
Thus $G$ is not cyclic and $G\cong C_2\times C_2$.
\end{proof}

\section{The irrational near-Tambara--Yamagami categories}\label{sec:irrational}
Throughout this section, $d\notin\Z$ and the notation of
Section~\ref{general-NTYC} is in force. For arguments depending only on
the fusion rules we work in a pseudo-unitary Galois conjugate. In
Theorem~\ref{thm:kappa-zero}, which concerns tensor equivalence, we
return to the original category.

\begin{proposition}\label{prop:rank-reduction}
One has $n=2$ and hence $\kappa=8k$ for some $k\in\Z_{\geq0}$.
\end{proposition}
\begin{proof}
The second summand in \eqref{eq:unit-dimensions} forgets to an object
of dimension $a+bd$ with $a,b\in\Z_{\geq0}$. Irrationality gives
$a=1$ and $b=\kappa/N$. Hence $\kappa=Nk$ with $k\geq0$ an integer.
The rational parts of the dimensions of the summands of $I(\one)$ are
$1,1,2,4$. Each forgetful image contains $\one$. If $n\geq4$, none
can contain $t$, contradicting $[F(I(\one)),t]=n$. Thus $n\leq3$.

If $n=3$, the unique summand of dimension $2+kd$ forgets to
$\one\oplus f\oplus k \rho$ for some $f\ne \one$. By \cite[Corollary~5.8]{Shimizu2017}, the fusion category $\mathcal C$ is unimodular. Hence, by
\cite[Theorem~4.10(6)]{Shimizu2017}, the right adjoint
$I$ of the forgetful functor $F$ satisfies
\[
I(\one)\cong I(\one)^*.
\]
Since the summand of dimension $2+kd$ occurs with multiplicity one
and this dimension is unique among the summands of $I(\one)$, that
summand is fixed by duality and is therefore self-dual. Thus $f=f^{-1}$,
impossible in a group of order nine. Therefore $n=2$.
\end{proof}

\subsection{The positive irrational parameter bound}\label{sec:bound}
The case where $k$ is $0$ will be discussed separately. All computations
in this subsection assume $k\geq1$.

By Proposition \ref{prop:rank-reduction}, we may write $\kappa=8k$ with $k\geq1$, and set
\begin{equation}\label{eq:normalization}
r=\sqrt{16k^2+8},\qquad q=2+kd,\qquad
d=4k+r,\quad d^2=8kd+8,\quad\Delta=8q.
\end{equation}

\begin{lemma}\label{lem:gram}
For $x\in G$ one has
\begin{align*}
 F(I(x))&\cong 4 x\oplus\bigoplus_{y\in G}2 y\oplus 2 t\oplus 8k \rho,\\
 F(I(t))&\cong \bigoplus_{y\in G}2 y\oplus 12 t\oplus 16k \rho,\\
 F(I(\rho))&\cong \bigoplus_{y\in G}8k y\oplus 16k t\oplus (16+64k^2) \rho.
\end{align*}
Write $G=\{\one,f,g,h\}$ and set $M_{X,Y}=[I(X),I(Y)]$.
In the order $\one,f,g,h,t,\rho$, this matrix is
\begin{equation}\label{eq:gram}
M=
\begin{pmatrix}
6&2&2&2&2&8k\\
2&6&2&2&2&8k\\
2&2&6&2&2&8k\\
2&2&2&6&2&8k\\
2&2&2&2&12&16k\\
8k&8k&8k&8k&16k&16+64k^2
\end{pmatrix}.
\end{equation}
\end{lemma}
\begin{proof}
Apply \eqref{eq:induction} to \eqref{eq:FI-general} and $M_{X,Y}=[I(X),I(Y)]=[FI(X),Y]$.
\end{proof}

\begin{lemma}\label{lem:unit}
After relabeling $f,g,h$, the six simple summands of $I(\one)$ are
$\one,A,B,C_0,C_g,C_h$, each of multiplicity one, with
\begin{equation}\label{eq:unit-table}
\begin{array}{c|c|l}
 Z&\dim Z&F(Z)\\ \hline
 \one&1&\one\\
 A&1+kd&\one\oplus k \rho\\
 B&q&\one\oplus f\oplus k \rho\\
 C_0&2q&\one\oplus f\oplus g\oplus h\oplus 2k \rho\\
 C_g&2q&\one\oplus g\oplus t\oplus 2k \rho\\
 C_h&2q&\one\oplus h\oplus t\oplus 2k \rho
\end{array}
\end{equation}
\end{lemma}
\begin{proof}
The ring $K_0(\C)$ is commutative. Its six characters consist of
$\chi_+,\chi_-$, the character with values $1$ on $G$, $-2$ on $t$,
and $0$ on $\rho$, and the three nontrivial characters of $G$
extended by zero on $t,\rho$. Their formal codegrees are
\[
 8+d^2,\quad 8+d_-^2,\quad 8,\quad 4,\quad 4,\quad 4.
\]
The formal-codegree theorem \cite[Theorem~2.13]{Ostrik} gives the
stated dimensions and multiplicities. In particular, every listed
forgetful image contains $\one$ exactly once. Irrationality gives
$F(A)\cong \one\oplus k \rho$ and $F(B)\cong \one\oplus f\oplus k \rho$ for some $f\ne \one$.

Each of the three objects of dimension $2q=4+2kd$ has $t$-multiplicity
at most one, since it already contains $\one$. The $t$-entry of
$F(I(\one))$ is $2$. Thus exactly two of these objects have images
\[
 \one\oplus x\oplus t\oplus 2k \rho,\qquad \one\oplus y\oplus t\oplus 2k \rho,
 \qquad x,y\in\{f,g,h\}.
\]
Each nonidentity element of $G$ has total multiplicity $2$ in
$F(I(\one))$. The case $x=y=f$ is therefore impossible. Up to interchanging
$g,h$ and $x,y$, the remaining possibilities are
$\{x,y\}=\{g,h\},\{f,g\},\{g,g\}$.

For $\{x,y\}=\{g,g\}$, the third image is
$\one\oplus f\oplus 2 h\oplus 2k \rho$. Subtracting the six known rows from the Gram
matrix gives, on the columns $f,g,h,t$,
\[
 \begin{pmatrix}
 4&2&0&2\\
 2&4&2&0\\
 0&2&2&2\\
 2&0&2&10
 \end{pmatrix}.
\]
Its determinant is $-16$, contradicting positive semidefiniteness.

For $\{x,y\}=\{f,g\}$, the third image is
$\one\oplus g\oplus 2 h\oplus 2k \rho$, and the residual Gram matrix is
\[
 \begin{pmatrix}
 4&2&2&1\\
 2&4&0&1\\
 2&0&2&2\\
 1&1&2&10
 \end{pmatrix}.
\]
These are inner products of nonnegative integer columns. The
$h$-column has exactly two nonzero entries, both $1$, and its support
is disjoint from that of the $g$-column. Its inner product with the
$f$-column is $2$. If the latter column had an entry $2$ on this
support, its squared norm $4$ would be exhausted, contradicting
its inner product $2$ with the $g$-column. The $f$-column therefore
equals $1$ at both entries of the $h$-column. Nonnegativity gives
$\langle f,t\rangle\geq\langle h,t\rangle=2$, contrary to
$\langle f,t\rangle=1$.

Only $\{x,y\}=\{g,h\}$ remains. The multiplicities in $F(I(\one))$
then determine the remaining image and give \eqref{eq:unit-table}.
\end{proof}

\begin{proposition}\label{prop:invertibles}
Every simple object of $\ZC$ whose
forgetful image contains an invertible object has twist $1$ or $-1$.
Besides the six summands of $I(\one)$, precisely two such objects have
twist $1$. They can be denoted by $P_1,P_2$, with
\begin{equation}\label{eq:positive-f}
 F(P_i)\cong f\oplus b_i \rho,\qquad b_1,b_2\in\Z_{\geq0},\qquad b_1+b_2=k.
\end{equation}
\end{proposition}
\begin{proof}
The part of $I(g)$ supported on $I(\one)$ is $C_0\oplus C_g$.
It has dimension $4q=\Delta/2$ and twist $1$. The complementary
part has the same dimension and, by \eqref{eq:indicators}, twist
trace $-4q$. Positivity of dimensions and the fact that twists
have absolute value one force every twist on the complement to
equal $-1$. The same argument applies to $I(h)$.
The pointed subcategory of a Tambara--Yamagami category is
$\operatorname{Vec}_G$ with trivial associator. With its canonical
positive spherical structure, $\nu_2(f)=1$. Hence
$\Tr(\theta^2_{I(f)})=\Delta=\dim I(f)$, so all twists on $I(f)$
are $\pm1$. The total dimension of its positive-twist part is
$\Delta/2=4q$. Its part in $I(\one)$ is $B\oplus C_0$, of dimension
$3q$; the residual positive-twist part therefore has dimension $q$.

Subtracting the rows in \eqref{eq:unit-table} from
\eqref{eq:gram} leaves, on $f,g,h,t$, the matrix
\begin{equation}\label{eq:residual}
 R=
 \begin{pmatrix}
 4&1&1&2\\
 1&4&1&1\\
 1&1&4&1\\
 2&1&1&10
 \end{pmatrix}.
\end{equation}
Every entry in each of the residual $f,g,h$ columns is $0$ or $1$.
Indeed, an entry at least $2$ would exhaust the squared norm $4$,
and then its inner product with another invertible column could
not be $1$.

Here and in what follows, ``contains'' refers to the underlying object
after applying the forgetful functor $F$; a simple summand of $I(f)$ is
itself an object of the Drinfeld center. More precisely, let $Z$ be a
residual simple summand of $I(f)$ with positive twist. If $g$ occurs in
$F(Z)$, then Frobenius reciprocity gives
\[
 [I(g),Z]=[F(Z),g]>0.
\]
Thus $Z$ is also a residual summand of $I(g)$, and the preceding analysis
of the complementary part of $I(g)$ says that its twist is $-1$, a
contradiction. The same argument excludes $h$. Consequently, among the
invertible constituents of a residual positive-twist summand, only $f$
can occur. The residual column calculation therefore writes its
forgetful image in the form
\[
 F(Z)\cong f\oplus u_Z t\oplus b_Z \rho .
\]
The total dimension of all residual positive-twist summands is
$2+kd$. Comparing integer parts, a sum of positive odd integers
$1+2u_Z$ equals $2$. Hence there are exactly two such summands, and both
have $u_Z=0$. Comparing the $d$-coefficients proves
\eqref{eq:positive-f}.
\end{proof}

\begin{notation}\label{not:UVW}
For each $Z\in\Irr(\ZC)$ write
\[
 F(Z)\cong \bigoplus_{x\in G}m_x(Z) x\oplus u_Z t\oplus b_Z \rho,\qquad
 s_Z=\sum_{x\in G}m_x(Z),\qquad a_Z=s_Z+2u_Z.
\]
Thus $\dim Z=a_Z+b_Zd$, with nonnegative integer coefficients.
Partition $\Irr(\ZC)$ into
\[
 \U=\{Z:s_Z>0\},\quad
 \V=\{Z:s_Z=0,\ u_Z>0\},\quad
 \W=\{Z:s_Z=u_Z=0\}.
\]
The summands in $\W$ have images $b_Z \rho$, with $b_Z>0$.
If $\mathscr S$ is one of these sets, $I(X)|_{\mathscr S}$
denotes the corresponding direct summand of $I(X)$.
\end{notation}

\begin{lemma}\label{lem:supports}
Set $T=\sum_{Z\in\U}u_Z^2$. Then $T\in\{4,8\}$.
More precisely, the $\D$-parts of the negative-twist objects of $\U$
are, as a multiset, one of the following:
\begin{equation}\label{eq:supports}
\begin{array}{c|l|c}
 &\D\text{-parts}&T\\ \hline
 \mathrm{I}& f\oplus g\oplus t,\ f\oplus h\oplus t,\ g\oplus h,\ g,\ g,\ h,\ h&4\\
 \mathrm{II}& f\oplus g\oplus h,\ f\oplus 2 t,\ g\oplus t,\ h\oplus t,\ g,\ g,\ h,\ h&8\\
 \mathrm{III}& f\oplus g\oplus h\oplus t,\ f\oplus t,\ g,\ g,\ g,\ h,\ h,\ h&4
\end{array}
\end{equation}
Repeated entries denote distinct central simple objects; all
$\rho$-terms have been suppressed.
\end{lemma}
\begin{proof}
Remove the two rows in \eqref{eq:positive-f} from
\eqref{eq:residual}. The remaining invertible columns have
$2,4,4$ entries equal to $1$, respectively; every pair has inner
product $1$. Their intersections either lie in a single row
containing $f,g,h$, or in three different rows containing
$f,g$, $f,h$, and $g,h$, respectively.

In the latter case the two rows containing $f$ have $t$-multiplicities
whose sum is $2$. Their individual multiplicities are at most $1$,
because the $gt$ and $ht$ inner products are both $1$. Both
multiplicities are thus $1$, and every other $t$-multiplicity in
these residual rows is $0$. This is pattern~I.

In the former case let $\alpha$ be the $t$-multiplicity in the
common triple row. The $gt$ and $ht$ inner products imply
$\alpha\in\{0,1\}$. The other row containing $f$ has
$t$-multiplicity $2-\alpha$. If $\alpha=0$, precisely one of the
three residual $g$-only rows and one of the three $h$-only rows
have $t$-multiplicity $1$, giving pattern~II. If $\alpha=1$,
all these multiplicities are $0$, giving pattern~III.

The rows $C_g,C_h$ in $I(\one)$ contribute $2$ to $T$, and the
objects $P_1,P_2$ contribute $0$. The three patterns contribute
$2,6,2$, respectively. This proves the displayed values.
\end{proof}

\begin{lemma}\label{lem:galois}
Choose a cyclotomic field containing the modular data of $\ZC$,
and an automorphism $\sigma$ whose restriction to $\Q(d)$ is
nontrivial. Let $\pi=\widehat{\sigma}$ be the induced permutation
of simple central objects. If $a_Z>0$, then
\begin{equation}\label{eq:reflection}
 m_x(\pi(Z))=m_x(Z),\quad u_{\pi(Z)}=u_Z,\quad
 b_{\pi(Z)}=ka_Z-b_Z.
\end{equation}
In particular, $\U$ and $\V$ are $\pi$-stable. Moreover,
\begin{equation}\label{eq:twist-galois}
 \theta_{\pi(Z)}=\sigma^2(\theta_Z)
\end{equation}
for every simple central $Z$. Thus $\theta_{\pi(Z)}=\theta_Z$
whenever $\theta_Z^8=1$.
\end{lemma}
\begin{proof}
Use the normalization $S=\Ss/\Delta$ from
Section~\ref{sec:preliminaries}. Galois symmetry gives signs
$\epsilon_Z\in\{\pm1\}$ such that
\begin{equation}\label{eq:S-galois}
 \sigma(S_{X,Z})=\epsilon_ZS_{X,\pi(Z)}.
\end{equation}
The twist identity \eqref{eq:twist-galois} holds for the canonical
modular representation of a Drinfeld center
\cite[Section~4.1 and Lemma~4.2]{DLN}.
One has $\sigma(d)=8k-d$ and
$\Delta/\sigma(\Delta)=1+kd$. The unit row of
\eqref{eq:S-galois} therefore gives
\[
 \epsilon_Z\dim\pi(Z)
 =(1+kd)\sigma(a_Z+b_Zd)=a_Z+(ka_Z-b_Z)d.
\]
If $a_Z>0$, the sign must be $+1$: otherwise comparison of the
integer coefficients of $1,d$ would give $a_{\pi(Z)}=-a_Z<0$.
Hence $a_{\pi(Z)}=a_Z$ and $b_{\pi(Z)}=ka_Z-b_Z$.

To recover the full $\D$-part, recall the character interpretation
of the summands $A_\psi$ of $I(\one)$
\cite[Section~2.3 and Theorem~2.13]{Ostrik}:
\[
 S_{A_\psi,Z}=\frac{\psi(F(Z))}{f_\psi}.
\]
For the character associated with $B$ this becomes
$S_{B,Z}=(s_Z-2u_Z)/8$. For each nontrivial character
$\eta:G\to\{\pm1\}$, it becomes
$S_{A_\eta,Z}=\frac14\sum_{x\in G}\eta(x)m_x(Z)$.
All four expressions are rational. Since $\epsilon_Z=1$,
\eqref{eq:S-galois} shows that they are unchanged under $\pi$.
Together with $a_Z=s_Z+2u_Z$, the first expression determines
$s_Z,u_Z$; Fourier inversion on $G$ then determines all $m_x(Z)$.
This proves \eqref{eq:reflection}. Finally every automorphism of
the field of eighth roots acts by an odd power, whose square is
$1$ modulo $8$, proving the last assertion.
\end{proof}

\begin{proposition}\label{prop:t}
With Notation~\ref{not:UVW},
\begin{equation}\label{eq:t-moments}
 \sum_{\U}u_Z^2=4,\qquad
 \sum_{\V}u_Z^2=8,\qquad
 \sum_{\V}u_Zb_Z=8k.
\end{equation}
Furthermore,
\begin{equation}\label{eq:t-traces}
 \nu_2(t)=1,\qquad
 \dim I(t)|_{\V}=\Delta,\qquad
 \Tr(\theta_{I(t)|_{\V}})=
 \Tr(\theta^2_{I(t)|_{\V}})=0.
\end{equation}
Every object of $\U\cup\V$ has twist of order dividing \(8\).
\end{proposition}
\begin{proof}
The matrix \eqref{eq:gram} gives $\sum_\U s_Zu_Z=8$.
By Galois reflection and $\pi$-stability of $\U$,
\[
 2\sum_\U u_Zb_Z
 =k\sum_\U u_Za_Z=k(8+2T).
\]
Consequently,
\begin{equation}\label{eq:U-t-dimension}
 \dim I(t)|_\U
 =\sum_\U u_Z(a_Z+b_Zd)
 =8+2T+(4k+kT)d=(4+T)q.
\end{equation}
The positive-twist part of $I(t)|_\U$ consists exactly of
$C_g,C_h$, and has dimension $4q$. Since $\dim I(t)=16q$,
\begin{equation}\label{eq:V-t-dimension}
 \dim I(t)|_\V=(12-T)q,\qquad
 \Tr(\theta_{I(t)|_\V})=(T-4)q.
\end{equation}
If $T=8$, the two quantities in \eqref{eq:V-t-dimension} both
equal $4q$, so every twist on $I(t)|_\V$ is $1$. All twists
on $I(t)|_\U$ already square to $1$. This gives
$\Tr(\theta^2_{I(t)})=2\Delta$, contradicting
$\nu_2(t)\in\{\pm1\}$. Lemma~\ref{lem:supports} therefore gives
$T=4$.

Now both parts of $I(t)$ have dimension $\Delta$, and its
$\V$-part has first twist trace $0$. Its $\U$-part contributes
$\Delta$ to the second trace. The value $\nu_2(t)=-1$ would
force the remaining trace to be $-2\Delta$, impossible for an
object of dimension $\Delta$. Thus $\nu_2(t)=1$ and the
remaining second trace is $0$.
The $(t,t)$ and $(t,\rho)$ entries of \eqref{eq:gram} now give
$\sum_\V u_Z^2=12-4=8$ and
$\sum_\V u_Zb_Z=16k-(4k+4k)=8k$.

The pivotal structure restricts to the canonical positive
structure on $\D$. Shimizu's formula
\cite[Theorem~3.4 and the formula immediately following it]{Shimizu}
gives
\[
 \nu_4(t)=\frac12\sum_{x\in G}\chi(x,x)^{-1}\in\{0,2\},
\]
where $\chi$ is the symmetric bicharacter defining $\D$.
Indeed $G$ has exponent $2$, so $x\mapsto\chi(x,x)$ is a
character with values in $\{\pm1\}$, and its sum is $0$ or $4$.
The $\U$-part of $I(t)$ contributes $\Delta$ to the fourth
trace. Hence
\[
 \Tr(\theta^4_{I(t)|_\V})=(\nu_4(t)-1)\Delta=\pm\Delta.
\]
As $\dim I(t)|_\V=\Delta$, all its fourth powers equal the
same sign. Thus $\theta_Z^8=1$ for $Z\in\V$; the assertion
for $\U$ follows from Proposition~\ref{prop:invertibles}.
\end{proof}

\begin{lemma}\label{lem:Q}
Let $Q=\sum_{Z\in\U}b_Z^2$. Then
\begin{equation}\label{eq:Q}
 \sum_\U a_Z^2=96,\qquad
 \sum_\U a_Zb_Z=48k,\qquad
 Q\geq\frac{49}{2}k^2.
\end{equation}
\end{lemma}
\begin{proof}
The invertible block of \eqref{eq:gram} gives
$\sum_\U s_Z^2=4\cdot6+12\cdot2=48$.
Also $\sum_\U s_Zu_Z=8$ and
$\sum_\U u_Z^2=4$ by Proposition~\ref{prop:t}.
Therefore
\[
 \sum_\U a_Z^2
 =\sum_\U(s_Z+2u_Z)^2=48+32+16=96.
\]
Lemma~\ref{lem:galois} implies
$2\sum_\U a_Zb_Z=k\sum_\U a_Z^2=96k$, and hence
\[
 Q=24k^2+\sum_\U\left(b_Z-\frac{k a_Z}{2}\right)^2.
\]
The objects $\one,A$ have $(a_Z,b_Z)=(1,0),(1,k)$,
respectively, and together contribute $k^2/2$ to the final sum.
\end{proof}

\begin{proposition}\label{prop:roots}
Put $p=b_1^2+b_2^2$, with $b_1,b_2$ as in
\eqref{eq:positive-f}, and let $\nu=\nu_2(\rho)\in\{\pm1\}$.
There is a multiset of roots of unity $\zeta_1,\ldots,\zeta_L$
such that
\begin{align}
 L&=16+56k^2-Q\leq16+\frac{63}{2}k^2,\label{eq:L}\\
 \sum_{i=1}^{L}\zeta_i
   &=Q-24k^2-2p-kr,\label{eq:first-sum}\\
 \sum_{i=1}^{L}\zeta_i^2
   &=24k^2-Q+(2\nu-6k)r.\label{eq:second-sum}
\end{align}
\end{proposition}
\begin{proof}
By Lemma~\ref{lem:galois} and Proposition~\ref{prop:t}, the
Galois reflection preserves twists on $\U\cup\V$. In particular,
it preserves each twist class in $\V$. Since $a_Z=2u_Z$ there,
for each root of unity $\zeta$ we have
\[
 \sum_{\substack{Z\in\V\\\theta_Z=\zeta}}u_Zb_Z
 =k\sum_{\substack{Z\in\V\\\theta_Z=\zeta}}u_Z^2.
\]
For $j=1,2$, the zero traces in \eqref{eq:t-traces} therefore give
\[
 0=\sum_\V u_Z(2u_Z+b_Zd)\theta_Z^j
   =(2+kd)\sum_\V u_Z^2\theta_Z^j.
\]
It follows that
\begin{equation}\label{eq:V-vanishing}
 \sum_\V u_Z^2\theta_Z^j
 =\sum_\V u_Zb_Z\theta_Z^j=0
 \qquad (j=1,2).
\end{equation}
For $Z\in\V$ put $\delta_Z=b_Z-ku_Z\in\Z$. Then
\begin{equation}\label{eq:V-delta}
 \frac1d\Tr(\theta^j_{I(\rho)|_\V})
 =\sum_\V b_Z^2\theta_Z^j
 =\sum_\V\delta_Z^2\theta_Z^j
 \qquad (j=1,2),
\end{equation}
where both equalities use \eqref{eq:V-vanishing}.

Take $b_Z^2$ copies of $\theta_Z$ for each $Z\in\W$ and
$\delta_Z^2$ copies for each $Z\in\V$. These nonnegative
integer multiplicities define the required auxiliary multiset.
Using the $\rho\rho$ entry of \eqref{eq:gram} and
\eqref{eq:t-moments}, its cardinality is
\begin{align*}
 L&=\sum_\W b_Z^2+\sum_\V(b_Z-ku_Z)^2\\
  &=16+64k^2-Q-2k(8k)+k^2(8)
    =16+56k^2-Q.
\end{align*}
Lemma~\ref{lem:Q} proves the inequality in \eqref{eq:L}.

The positive-twist objects of $\U$ are precisely
$\one,A,B,C_0,C_g,C_h,P_1,P_2$. Their $a_Z^2$-sum is
$1+1+4+16+16+16+1+1=56$.
Thus $\sum_\U\theta_Za_Z^2=56-(96-56)=16$.
Twist-preserving Galois reflection implies
\[
 \sum_\U\theta_Za_Zb_Z
 =\frac{k}{2}\sum_\U\theta_Za_Z^2=8k.
\]
The positive-twist $b_Z^2$-sum is $14k^2+p$, so
$\sum_\U\theta_Zb_Z^2=28k^2+2p-Q$.
Consequently the first trace identity for $I(\rho)$ is
\[
 0=8k+(28k^2+2p-Q)d+d\sum_{i=1}^L\zeta_i.
\]
Since $8/d=d-8k$ and $d=4k+r$, this is
\eqref{eq:first-sum}.

For the second trace, every twist in $\U$ squares to $1$.
Lemma~\ref{lem:Q} and \eqref{eq:V-delta} give
\[
 \nu\Delta=48k+Qd+d\sum_{i=1}^L\zeta_i^2.
\]
Now $\Delta/d=2r$ and $48k/d=6kr-24k^2$, yielding
\eqref{eq:second-sum}.
\end{proof}

For a positive integer
$M$, we write
\[
 v_2(M)=\max\{r\in\mathbb Z_{\geq0}:2^r\mid M\}.
\]
Thus $v_2(M)$ is the exponent of $2$ in the prime factorization of $M$.

\begin{lemma}\label{lem:root-bound}
Let $c>1$ be an even squarefree integer. Suppose roots of unity
$\zeta_1,\ldots,\zeta_L$ satisfy
\[
 \sum_{i=1}^L\zeta_i=A+B\sqrt c,\qquad
 \sum_{i=1}^L\zeta_i^2=A'+B'\sqrt c,
 \qquad A,B,A',B'\in\Z.
\]
Then
\begin{equation}\label{eq:joint-bound}
 L\geq\varphi(2c)(|B|+|B'|).
\end{equation}
\end{lemma}
\begin{proof}
Let $K=\Q(\sqrt c)$ and choose a cyclotomic field $E$ containing
$K$ and all the roots in question. Write
$H=\Gal(E/K)$ and $h=|H|$. The conductor of $K$ is $4c$. Since $c$ is even and squarefree,
we can write $c=2c_0$ with $c_0$ odd; hence
$4c=8c_0$ and
\[
 v_2(4c)=v_2(8c_0)=3.
\]

For a root $\zeta$ of exact order $n$, consider
$\mathcal T(\zeta)=\sum_{\tau\in H}\tau(\zeta)\in K$.
If $K\not\subseteq\Q(\zeta)$, then
$K\cap\Q(\zeta)=\Q$. The restrictions of $H$ to $\Q(\zeta)$
run through its full Galois group with equal multiplicity, so
$\mathcal T(\zeta)$ is rational.
If $K\subseteq\Q(\zeta)$, then $4c\mid n$. When $v_2(n)\geq4$,
the automorphism
$\zeta\mapsto\zeta^{1+n/2}=-\zeta$ fixes $K$, since
$4c\mid n/2$. It extends to an element of $H$, and hence
$\mathcal T(\zeta)=0$.
Thus a nonzero $\sqrt c$-coefficient in $\mathcal T(\zeta)$
is possible only when $v_2(n)=3$.

Let $J_3=\{i:v_2(\ord\zeta_i)=3\}$ and
$J_4=\{i:v_2(\ord\zeta_i)=4\}$.
Since $\ord(\zeta_i^2)=\ord(\zeta_i)/\gcd(2,\ord\zeta_i)$,
the preceding observation gives
\[
 \sum_{i\in J_3}\mathcal T(\zeta_i)=a+hB\sqrt c,\qquad
 \sum_{i\in J_4}\mathcal T(\zeta_i^2)=a'+hB'\sqrt c
\]
for some integers $a,a'$. Indeed these expressions are
algebraic integers in $K$, and the ring of integers of $K$
is $\Z[\sqrt c]$.

By \cite[Theorem~6.2]{Larson}, writing an element
$a+b\sqrt c$ as a sum of roots of unity requires at least
$|b|\varphi(2c)$ roots. The two displayed expressions are sums
of $h|J_3|$ and $h|J_4|$ roots, respectively. Therefore
$|J_3|\geq|B|\varphi(2c)$ and
$|J_4|\geq|B'|\varphi(2c)$.
The two index sets are disjoint, and \eqref{eq:joint-bound}
follows.
\end{proof}

\begin{theorem}\label{thm:main}
If $\C=\C(\D,\kappa)$ is a near-Tambara--Yamagami fusion
category with $d_+\notin\Z$ and $\kappa>0$, then
$n=2$ and $\kappa=8$.
\end{theorem}
\begin{proof}
By Proposition~\ref{prop:rank-reduction}, $\kappa=8k$ with
$k\in\Z_{\geq1}$ and $n=2$. Write
$r=m\sqrt c$, where $m\geq1$ is an integer and $c$ is squarefree.
Since $r^2=8(2k^2+1)$ has $2$-adic valuation $3$, $c$ is even.
Apply Lemma~\ref{lem:root-bound} to
\eqref{eq:first-sum}--\eqref{eq:second-sum}.
As $\nu\in\{\pm1\}$ and $k\geq1$,
\[
 L\geq m\varphi(2c)\bigl(k+|2\nu-6k|\bigr)
   =(7k-2\nu)m\varphi(2c)
   \geq(7k-2)m\varphi(2c).
\]
Together with \eqref{eq:L}, this gives
\begin{equation}\label{eq:final-bound}
 16+\frac{63}{2}k^2
 \geq(7k-2)\sqrt{16k^2+8}\,
       \frac{\varphi(2c)}{\sqrt c}.
\end{equation}
For even squarefree $c$,
\[
 \frac{\varphi(2c)}{\sqrt c}
 =\sqrt2\prod_{\substack{\ell\mid c\\\ell\ {\rm odd\ prime}}}
            \frac{\ell-1}{\sqrt\ell}
 \geq\sqrt2.
\]
Since $\sqrt{16k^2+8}>4k$ and $\sqrt2>7/5$, the right side
of \eqref{eq:final-bound} is greater than
$\frac{28}{5}k(7k-2)$. But
\[
 \frac{28}{5}k(7k-2)-\left(16+\frac{63}{2}k^2\right)
 =\frac{77k^2-112k-160}{10}>0
 \qquad(k\geq3).
\]
The last quadratic is positive at $3$ and increasing thereafter.
Thus $k\leq2$.

If $k=2$, then $r=6\sqrt2$, so $m=6$, $c=2$, and
$\varphi(2c)=2$. The lower bound gives $L\geq144$,
whereas \eqref{eq:L} gives $L\leq142$. This contradiction
excludes $k=2$, proving $k=1$ and hence $\kappa=8$.
\end{proof}

Together with the case $\kappa=0$, the numerical conclusion is
$d_+\notin\Z\Longrightarrow (n,\kappa)\in\{(2,0),(2,8)\}$.

\subsection{The zero-parameter case}

\begin{theorem}\label{thm:kappa-zero}
Let $\C=\C(\D,0)$ be a near-Tambara--Yamagami fusion category
over $\CC$, with $\D$ nonpointed. Then
$\D\simeq\Rep(H_8)$, where $H_8$ is the Kac--Paljutkin Hopf
algebra, and $\C$ is tensor equivalent to a category in the
eight-entry list of rank-one odd-component extensions in
\cite[Section~7, Tables~3--4, and Appendix~C.1]{CLS}.
\end{theorem}
\begin{proof}
Since $d=n\sqrt2$ is irrational, Proposition~\ref{prop:rank-reduction}
gives $n=2$ and hence $\FPdim(\D)=8$. Put
$\mathcal M=\operatorname{add}(\rho)$. The fusion rules give
\[
 \C=\D\oplus\mathcal M,\qquad
 \D\otimes\mathcal M,\ \mathcal M\otimes\D\subseteq\mathcal M,
 \qquad \mathcal M\otimes\mathcal M\subseteq\D.
\]
Thus $\C$ is faithfully graded by $\mathbb Z/2\mathbb Z$ with
neutral component $\D$ and a rank-one odd component. As in the proof
of Lemma~\ref{lem:order-bound}, the action on $\mathcal M$ gives the
fiber functor
\[
 \Phi:\D\longrightarrow\End_{\CC}(\mathcal M)\simeq\Vect,
 \qquad \Phi(X)=\Hom_{\C}(\rho,X\otimes\rho).
\]
Here $\End_{\CC}(\mathcal M)$ denotes ordinary complex-linear
endofunctors, with composition as tensor product. Reconstructing from
this actual fiber functor gives a semisimple Hopf algebra $H$ such that
$\D\simeq\Rep(H)$ and $\dim H=8$.

The odd component $\mathcal M$ is an invertible $\D$-bimodule
category. Consequently its right action induces an equivalence
\[
 \D^{\mathrm{op}}\simeq\End_{\D}(\mathcal M),
\]
where $\mathrm{op}$ reverses the tensor product; see
\cite[Theorem~6.1 and Proposition~4.2]{ENOH}.
The fiber-functor duality of \cite[Example~7.12.26]{EGNO}, with the
correction recorded in \cite{EGNOcorr}, identifies this module
endofunctor category with $\Rep(H^{*\mathrm{cop}})$ in the
composition convention. In particular it has the same simple
dimensions as $\Rep(H^*)$. Since $\D$ is nonpointed, both $H$ and
$H^*$ are noncommutative as algebras. Thus $H$ is neither commutative
nor cocommutative. Masuoka's dimension-eight classification
\cite{Masuoka} now gives $H\cong H_8$, so $\D\simeq\Rep(H_8)$.

To apply the unitary calculation in \cite{CLS}, observe directly from
the fusion rules that
\[
 \C_{\mathrm{ad}}=\D,\qquad
 \D_{\mathrm{ad}}=\D_{\mathrm{pt}},\qquad
 (\D_{\mathrm{pt}})_{\mathrm{ad}}=\Vect.
\]
Thus $\C$ is nilpotent, hence weakly group-theoretical
\cite{ENO}. By \cite[Theorem~5.20]{GHR}, both $\C$ and $\D$ are completely
unitary. In particular, $\C$ is unitarizable, and the identification
of its neutral category with $\Rep(H_8)$ can be chosen unitary;
see \cite[Definition~5.13(1)]{GHR}.
After choosing the identification of the neutral component with
$\Rep(H_8)$ and the labels of its simples, its associator is therefore
covered by the unitary pentagon calculation of
\cite[Section~7, Tables~3--4, and Appendix~C.1]{CLS}. That calculation
lists eight solutions modulo changes of bases in the fusion spaces,
with the neutral category fixed. Forgetting these choices gives the
claimed membership in the eight-entry list. We do not need to decide
here which entries remain distinct under tensor equivalences that
relabel the neutral component.
\end{proof}

\subsection{Nonexistence for parameter eight}\label{r8:sec:induction}

From now on suppose, for contradiction, that $\C=\C(\D,8)$ exists. The subcategory $\D$ generated by $G$ and $t$ is integral of dimension $8$. By \cite[Lemma~4.1]{DS}, a Galois
conjugate of $\cC$ is pseudo-unitary and has the same fusion rules.
Replace $\cC$ by this conjugate and equip it with its canonical spherical
structure. Every dimension below is therefore positive and equals
Frobenius--Perron dimension. Put
\begin{equation}\label{r8:eq:arithmetic}
s=\sqrt6,\qquad d=4+2s,\qquad
D=48+16s=8(d+2),\qquad \varepsilon=1+d.
\end{equation}
The identities
\begin{equation}\label{r8:eq:arithmetic-identities}
d^2=8+8d,\qquad \frac Dd=4s,\qquad
D'=48-16s,\qquad \frac D{D'}=\varepsilon
\end{equation}
will be used repeatedly. The matrix \eqref{eq:gram} now is
\begin{equation}\label{r8:eq:M}
M=\begin{pmatrix}
6&2&2&2&2&8\\
2&6&2&2&2&8\\
2&2&6&2&2&8\\
2&2&2&6&2&8\\
2&2&2&2&12&16\\
8&8&8&8&16&80
\end{pmatrix}.
\end{equation}

Here $D=\Delta=\dim(\C)$. We use the unnormalized matrix $\Ss$,
the normalized matrix $S=\Ss/D$, and the character identity
\eqref{eq:central-character} fixed in Section~\ref{sec:preliminaries}.

The Ng--Schauenburg indicator formula \cite[Theorem~4.1]{NS} is
\begin{equation}\label{r8:eq:indicator}
\sum_{Z\in\Irr(\cZ)}m_X(Z)\dim(Z)\theta_Z^n=D\nu_n(X).
\end{equation}
For a nontrivial simple $X$, $\nu_1(X)=0$. All simples of $\cC$ are
self-dual, so $\nu_2(X)\in\{1,-1\}$. The fourth indicator is the trace
of an operator of order four on $\Hom(\one,X^{\otimes4})$, hence
\begin{equation}\label{r8:eq:fourth-cyclotomic}
\nu_4(X)\in\ZZ[\ii].
\end{equation}
The balancing identity is \eqref{eq:balancing-general}.

\paragraph{Characters and the pointed subcategory}

The characters of $R_8$ are the two dimension characters $\chi_\pm$,
the character $\eta$, and the three nontrivial characters $\lambda$ of
$G$, extended as follows:
\begin{equation}\label{r8:eq:characters}
\begin{array}{c|ccc|c}
 &x\in G&t&\rho& f_\chi\\ \hline
\chi_+&1&2&4+2s&D\\
\chi_-&1&2&4-2s&D'\\
\eta&1&-2&0&8\\
\lambda&\lambda(x)&0&0&4.
\end{array}
\end{equation}
The table follows directly from \eqref{eq:fusion}. Its six distinct rows
exhaust the characters of the complexified rank-six ring.

\begin{lemma}\label{r8:lem:pointed}
The pointed subcategory of $\cC$ is tensor equivalent to $\Vect_G$.
For every $x\in G$, one has $\nu_2(x)=1$. There is at most one central
invertible object with prescribed restriction $x\in G$.
\end{lemma}
\begin{proof}
The additive subcategory generated by $\rho$ is a rank-one left module
category over the pointed subcategory, since $x\otimes\rho\cong \rho$ for all $x\in G$.
Choosing isomorphisms $x\otimes\rho\cong\rho$, its module associativity says that
the pointed associator is a coboundary. Thus the pointed subcategory is
$\Vect_G$. The positive spherical structure restricts to the canonical
one, and every element of $G$ has second indicator $1$.

A tensor automorphism of $\operatorname{Id}_{\cC}$ acts by a scalar $a_X$
on each simple $X$, with $a_Xa_Y=a_Z$ whenever $Z$ occurs in $X\otimes Y$.
The product $\rho\otimes \rho$ contains $\rho$, so $a_\rho^2=a_\rho$ and
$a_\rho=1$. It also contains every other simple, so all $a_X=1$.
A central structure on the unit is the same as such a tensor
automorphism. If two central invertibles had the same restriction,
their quotient would give a nontrivial central structure on the unit.
This is impossible.
\end{proof}

\begin{lemma}\label{r8:lem:signs}
For $x\in G\setminus\{1\}$, every simple summand of $I(x)$ has twist
$1$ or $-1$. Counting induction multiplicities, each of these two twist
sectors has dimension $D/2$.
\end{lemma}
\begin{proof}
By Lemma~\ref{r8:lem:pointed} and \eqref{r8:eq:indicator},
\[
\sum_Zm_x(Z)\dim Z\,\theta_Z^2=D
=\sum_Zm_x(Z)\dim Z.
\]
Every summand has a positive dimension and $|\theta_Z|=1$. Equality in
the triangle inequality therefore forces $\theta_Z^2=1$ whenever
$m_x(Z)>0$. The first indicator of $x$ vanishes, so the two twist
sectors have equal dimension.
\end{proof}

Call a simple central object \emph{low} if its restriction contains
$\one$, or equivalently if it occurs in $I(\one)$. This terminology will
only be used in this section.

By \eqref{r8:eq:characters} and \eqref{eq:central-character},
\begin{equation}\label{r8:eq:Ione}
I(\one)\cong \one\oplus A\oplus B\oplus\bigoplus_{\lambda\ne1}C_\lambda,
\end{equation}
where $A,B,C_\lambda$ correspond to $\chi_-,\eta,\lambda$, respectively.
Their dimensions are
\begin{equation}\label{r8:eq:low-dimensions}
\dim A=1+d,\qquad \dim B=d+2,\qquad
\dim C_\lambda=2(d+2).
\end{equation}
All of them have twist $1$.

The coefficient of $d$ in a restriction dimension is its
$\rho$-multiplicity. Since each low object contains $\one$ once,
irrationality gives
\begin{equation}\label{r8:eq:AB}
F(A)\cong \one\oplus \rho,\qquad F(B)\cong \one\oplus g\oplus \rho
\end{equation}
for some nontrivial $g\in G$. Fix this $g$ and label the other elements
$h,k=gh$. These are the labels $f,g,h$, respectively, from
Section~\ref{sec:bound}; this local relabeling is used throughout the
$R_8$ obstruction. A restriction $FC_\lambda$ contains $\rho$ twice; its integral
part has dimension $4$. Set $b_\lambda=\ip{FC_\lambda}{t}$.
Symmetry of $\Ss$, applied to $B,C_\lambda$, gives
\[
2(d+2)(1+\lambda(g))=(d+2)(4-4b_\lambda).
\]
Thus
\begin{equation}\label{r8:eq:b-lambda}
b_\lambda=\frac{1-\lambda(g)}2.
\end{equation}
Precisely one of the three objects, denoted $C_0$, has $b_\lambda=0$;
the other two have $b_\lambda=1$.

By the first row of \eqref{r8:eq:M}, the low contributions to $I(g),I(h)$,
and $I(k)$ have dimensions
\begin{equation}\label{r8:eq:low-masses}
3(d+2),\qquad4(d+2),\qquad4(d+2),
\end{equation}
respectively. For $g$, the contribution is $B$ together with one
$C_\lambda$, counted with multiplicity; for $h$ and $k$ it consists of
two $C_\lambda$'s, counted with multiplicity. In particular,
Lemma~\ref{r8:lem:signs} implies that every nonlow central simple whose
restriction contains $h$ or $k$ has twist $-1$.

\begin{lemma}\label{r8:lem:first-ten}
There is a central invertible $P$ with $F(P)\cong g$, $P\otimes P\cong \one$, and
$\theta_P=1$. After naming the two remaining low objects $C_h,C_k$, their
restrictions are
\begin{align}
F(C_0)&\cong \one\oplus g\oplus h\oplus k\oplus 2 \rho,\label{r8:eq:Czero}\\
F(C_h)&\cong \one\oplus h\oplus t\oplus 2 \rho,\qquad
F(C_k)\cong \one\oplus k\oplus t\oplus 2 \rho.\label{r8:eq:Chk}
\end{align}
Moreover, $P\otimes B\cong B$, $P\otimes C_0\cong C_0$, and the six summands of $I(g)$ are
\[
B,\ C_0,\ P,\ P\otimes A,\ P\otimes C_h,\ P\otimes C_k,
\]
each with multiplicity one. The first four have twist $1$ and the last
two have twist $-1$.
\end{lemma}
\begin{proof}
By \eqref{r8:eq:low-masses} and Lemma~\ref{r8:lem:signs}, the nonlow twist-$1$
part of $I(g)$ has dimension $d+2$. Its simple constituents cannot contain
$h$ or $k$ in their restrictions. Write their restrictions as
$r g\oplus b t\oplus c \rho$, with $r\ge1$, $b,c\ge0$. Since each occurs in $I(g)$
with multiplicity $r$, comparison of the coefficients of $1,d$ gives
\begin{equation}\label{r8:eq:plus-equations}
\sum rc=1,\qquad \sum(r^2+2rb)=2.
\end{equation}
The first equation gives exactly one constituent with $r=c=1$; its
$b$ is zero by the second equation. The remaining contribution to the
second equation is $1$, so there is exactly one further constituent, with
$(r,b,c)=(1,0,0)$. Denote the latter by $P$. It has dimension one,
$FP\cong g$, and twist $1$. Lemma~\ref{r8:lem:pointed} implies $P\otimes P\cong \one$.
Let $Z$ be the former constituent, with $F(Z)\cong g\oplus\rho$.
The simple object $P\otimes Z$ forgets to $\one\oplus\rho$, so by
adjunction it occurs in $I(\one)$. The unique summand of $I(\one)$ of
dimension $1+d$ is $A$. Hence $P\otimes Z\cong A$, and $P^2\cong\one$
gives $Z\cong P\otimes A$.

The restriction of $P\otimes B$ equals $FB$ and contains $\one$. Among the low
objects only $B$ has dimension $d+2$, hence $P\otimes B\cong B$. Equations
\eqref{eq:central-character} and \eqref{eq:balancing-general} give
\begin{equation}\label{r8:eq:Ptwist}
\theta_{P\otimes C_\lambda}=\lambda(g).
\end{equation}
If $F(C_0)$ did not contain $g$, then $P\otimes C_0$ would be nonlow with twist
$1$. Its restriction would contain $h$ or $k$, because the integral part
of $F(C_0)$ has dimension $4$, contains $\one$ once, and has no $t$.
This contradicts the observation after \eqref{r8:eq:low-masses}. Thus
$P\otimes C_0$ is low. It has no $t$ in its restriction, so $P\otimes C_0\cong C_0$ by
\eqref{r8:eq:b-lambda}. Invariance under $g$ now forces
$\ip{F(C_0)}{g}=1$ and equal coefficients on $h,k$; the integral dimension
is $4$, proving \eqref{r8:eq:Czero}.

The first row of \eqref{r8:eq:M} now forces the two remaining restrictions
to be those in \eqref{r8:eq:Chk}. Their $P$-translates have restrictions
\begin{equation}\label{r8:eq:PChk}
F(P\otimes C_h)\cong g\oplus k\oplus t\oplus 2 \rho,\qquad
F(P\otimes C_k)\cong g\oplus h\oplus t\oplus 2 \rho
\end{equation}
and twist $-1$ by \eqref{r8:eq:Ptwist}. These, together with $B,C_0,P,P\otimes A$,
are six distinct simples containing $g$ once. Since $M_{g,g}=6$, they
exhaust $I(g)$.
\end{proof}

\begin{table}[tb]
\centering
\caption{All central simples whose restrictions meet $G$. Here
$\varepsilon=1+d$.}
\label{r8:tab:fifteen}
\renewcommand{\arraystretch}{1.12}
\begin{tabular}{@{}llll@{}}
\toprule
$Z$&$FZ$&$\dim Z$&$\theta_Z$\\
\midrule
$\one$&$\one$&$1$&$1$\\
$P$&$g$&$1$&$1$\\
$Q$&$h$&$1$&$-1$\\
$P\otimes Q$&$k$&$1$&$-1$\\
$A$&$\one\oplus \rho$&$\varepsilon$&$1$\\
$P\otimes A$&$g\oplus \rho$&$\varepsilon$&$1$\\
$Q\otimes A$&$h\oplus \rho$&$\varepsilon$&$-1$\\
$P\otimes Q\otimes A$&$k\oplus \rho$&$\varepsilon$&$-1$\\
$B$&$\one\oplus g\oplus \rho$&$d+2$&$1$\\
$Q\otimes B$&$h\oplus k\oplus \rho$&$d+2$&$-1$\\
$C_0$&$\one\oplus g\oplus h\oplus k\oplus 2 \rho$&$2(d+2)$&$1$\\
$C_h$&$\one\oplus h\oplus t\oplus 2 \rho$&$2(d+2)$&$1$\\
$C_k$&$\one\oplus k\oplus t\oplus 2 \rho$&$2(d+2)$&$1$\\
$P\otimes C_h$&$g\oplus k\oplus t\oplus 2 \rho$&$2(d+2)$&$-1$\\
$P\otimes C_k$&$g\oplus h\oplus t\oplus 2 \rho$&$2(d+2)$&$-1$\\
\bottomrule
\end{tabular}
\end{table}

\begin{proposition}\label{r8:prop:fifteen}
There is a central invertible object $Q$ with $FQ\cong h$. The fifteen simple objects
in Table~\ref{r8:tab:fifteen} are exactly the central simples whose
restrictions contain an invertible object of $\cC$. The group of central
invertibles is $\{\one,P,Q,P\otimes Q\}\cong C_2^2$.
\end{proposition}
\begin{proof}
Consider the ten objects consisting of the six in \eqref{r8:eq:Ione} and
$P,P\otimes A,P\otimes C_h,P\otimes C_k$. Lemma~\ref{r8:lem:first-ten} exhausts both $I(\one)$ and
$I(g)$. The four objects $C_h,C_k,P\otimes C_h,P\otimes C_k$ already exhaust the mixed Gram
entries $M_{h,t}=M_{k,t}=2$.
Every further simple meeting $h$ or $k$ consequently has restriction
\[
a h\oplus b k\oplus c \rho.
\]
Subtracting the contributions of the ten known objects from
\eqref{r8:eq:M} gives
\begin{equation}\label{r8:eq:hk-residual}
\sum a^2=\sum b^2=3,\qquad
\sum ab=1,\qquad
\sum ac=\sum bc=2.
\end{equation}
The first two square sums imply $a,b\in\{0,1\}$. Thus there is one
simple containing both $h,k$, two containing only $h$, and two containing
only $k$. Let $z$ be the $\rho$-multiplicity of the first of these;
then $z\in\{0,1,2\}$.

If $z=2$, both $h$-only objects have restriction $h$, contrary to the
uniqueness in Lemma~\ref{r8:lem:pointed}. Suppose $z=0$ and denote the
object with restriction $h\oplus k$ by $E$. Uniqueness of the object containing
both $h,k$ gives $E^*\cong E$ and $P\otimes E\cong E$. The restriction of $E\otimes E$ is
$2 (\one\oplus g)$. Among the already exhausted summands of $I(\one)$ and
$I(g)$, only $\one,P$ have restrictions supported on $\{\one,g\}$.
They are therefore the only possible constituents of $E\otimes E$.
Reciprocity gives each multiplicity one, contradicting
$\dim(E\otimes E)=4$.

Hence $z=1$. The two $h$-only objects have respective restrictions
$h$ and $h\oplus \rho$. The former is an invertible $Q$, and the latter is
$Q\otimes A$. Their $P$-translates give the two $k$-only objects. The remaining
object is $Q\otimes B$, with restriction $h\oplus k\oplus \rho$. All five have twist $-1$
by \eqref{r8:eq:low-masses}. Their restrictions distinguish them and
exhaust the possibilities in \eqref{r8:eq:hk-residual}.
Lemma~\ref{r8:lem:pointed} gives $Q\otimes Q\cong \one$ and uniqueness of the four
central invertibles; they commute because the Drinfeld center $\ZC$ is braided.
This proves the table and the assertion of exhaustiveness.
\end{proof}

Let $\mathcal R$ be the set of central simples not in
Table~\ref{r8:tab:fifteen}. By Proposition~\ref{r8:prop:fifteen}, each has
restriction
\begin{equation}\label{r8:eq:residual-restriction}
FZ\cong b_Z t\oplus c_Z \rho,\qquad b_Z,c_Z\in\ZZ_{\ge0}.
\end{equation}
Subtracting the fifteen restriction vectors from \eqref{r8:eq:M} gives
\begin{equation}\label{r8:eq:residual-gram}
\sum_{Z\in\mathcal R}b_Z^2=8,\qquad
\sum_{Z\in\mathcal R}b_Zc_Z=8,\qquad
\sum_{Z\in\mathcal R}c_Z^2=54.
\end{equation}

\paragraph{Dimension restrictions}

\begin{lemma}\label{r8:lem:four-rays}
For every $Z\in\mathcal R$, its restriction lies on one of the four rays
\[
b t,\qquad b (t\oplus\rho),\qquad b (t\oplus 2 \rho),\qquad c \rho,
\]
with positive integer coefficient on the indicated ray.
\end{lemma}
\begin{proof}
If $b_Z=0$, the assertion is immediate. Otherwise write $b=b_Z$,
$c=c_Z$; then $b\in\{1,2\}$ by \eqref{r8:eq:residual-gram}.
Let $\sigma$ be a cyclotomic field automorphism restricting to the
nontrivial automorphism of $\QQ(s)$. Signed Galois symmetry of the
normalized matrix $\Ss/D$ gives a permutation $\widehat\sigma$ of
central simples satisfying
\begin{equation}\label{r8:eq:galois-dimension}
\dim\widehat\sigma(Z)=\frac D{D'}\,|\sigma(\dim Z)|
=\left|\varepsilon\sigma(2b+cd)\right|
=\left|2b+(2b-c)d\right|.
\end{equation}
See \cite[Section~4.1, equation~(4.3)]{DLN}; the absolute value records
the sign in the signed permutation, since all dimensions are positive.
If $c>2b$, then $2b+(2b-c)d<0$, because $d>4\ge2b$.
Its absolute value has rational coefficient $-2b$ in the basis $1,d$.
But the dimension of any restriction to $\cC$ has nonnegative
coefficients in this basis. Hence $0\le c\le2b$.

For every central simple $Z$, the number $D/\dim Z$ is an algebraic
integer. Indeed, the formal codegree attached to $Z$ in the modular
category $\cZ$ is
\[
 f_Z=\frac{\FPdim(\cZ)}{\dim(Z)^2}=\left(\frac D{\dim Z}\right)^2.
\]
Formal codegrees are algebraic integers. Thus $D/\dim Z$ satisfies the
monic equation $X^2-f_Z=0$ over the ring of algebraic integers, and
transitivity of integrality shows that it is itself an algebraic integer.
If $b=2$, the possibilities $c=1,3$ would give, respectively,
\[
\frac D{4+d}=\frac{24+4s}{5},\qquad
\frac D{4+3d}=\frac{24-4s}{5}.
\]
Both have nonintegral field trace $48/5$, so neither is an algebraic
integer. The remaining pairs $(b,c)$ give exactly the four stated rays.
\end{proof}

\begin{lemma}\label{r8:lem:U}
There is at most one central simple $U$ whose restriction is supported
on $t$; if it exists, $F(U)\cong t$. There is then exactly one residual simple
with restriction on the ray $b (t\oplus 2 \rho)$, namely $U\otimes A$, and
\[
F(U\otimes A)\cong t\oplus 2 \rho,\qquad \theta_{U\otimes A}=\theta_U.
\]
If $U$ does not exist, there are no residual simples on that ray.
\end{lemma}
\begin{proof}
Suppose $FU\cong b t$. The restriction of $U\otimes U^*$ is supported on $G$.
By Proposition~\ref{r8:prop:fifteen}, every simple constituent is one of
$\one,P,Q,P\otimes Q$. Each has multiplicity at most one, since for an invertible
$L$,
\[
[U\otimes U^*,L]=\dim\Hom(L\otimes U,U)\le1.
\]
Consequently $4b^2=\dim(U\otimes U^*)\le4$, so $b=1$ and
\begin{equation}\label{r8:eq:UUstar}
U^*\otimes U\cong U\otimes U^*\cong \one\oplus P\oplus Q\oplus (P\otimes Q).
\end{equation}
If distinct $U,V$ both restricted to $t$, then $U\otimes V^*$ would be supported
on the same four central invertibles, each with multiplicity at most one,
but would not contain $\one$. This contradicts $\dim(U\otimes V^*)=4$.

The dimension transformation in \eqref{r8:eq:galois-dimension} interchanges
$2b$ and $2b+2bd$. None of these dimensions, for $b=1,2$, occurs in
Table~\ref{r8:tab:fifteen}. A simple of either dimension must therefore be
residual, and irrationality identifies its restriction as $b t$ or
$b (t\oplus 2 \rho)$. The Galois permutation gives equal numbers of objects of
each paired dimension. Thus there is exactly one object on the latter
ray when $U$ exists, and none otherwise.

To identify that object, reciprocity and \eqref{r8:eq:UUstar} give
\[
\dim\End(U\otimes A)=\sum_{L\in\{\one,P,Q,P\otimes Q\}}\dim\Hom(A,L\otimes A)=1.
\]
Indeed, the four restrictions $F(L\otimes A)\cong FL\oplus \rho$ are distinct. Therefore
$U\otimes A$ is simple and $F(U\otimes A)\cong t\oplus 2 \rho$. Finally
\eqref{eq:central-character}, for the character $\chi_-$ of $A$, gives
$\Ss_{A,U}=2\varepsilon=\dim(U\otimes A)$. Since $U\otimes A$ is simple and
$\theta_A=1$, the balancing identity \eqref{eq:balancing-general} implies
$\theta_{U\otimes A}=\theta_U$.
\end{proof}

\paragraph{A multiset of forty-six twists}

Let $x\in\{0,1\}$ record whether $U$ exists. Denote the remaining
residual simples by $V$ and $W$ according to their restrictions
\[
FV\cong b_V (t\oplus\rho),\qquad FW\cong c_W \rho.
\]
Equations \eqref{r8:eq:residual-gram} and Lemma~\ref{r8:lem:U} become
\begin{equation}\label{r8:eq:counts}
\sum_Vb_V^2=8-2x,\qquad
\sum_Wc_W^2=46-2x.
\end{equation}
Define a multiset $\Lambda$ by taking $c_W^2$ copies of $\theta_W$ for
each $W$, and, when $x=1$, two additional copies of $\theta_U$.
In particular,
\begin{equation}\label{r8:eq:size46}
|\Lambda|=46.
\end{equation}
Put $p_n=\sum_{\zeta\in\Lambda}\zeta^n$.

\begin{proposition}\label{r8:prop:moments}
One has $\nu_2(t)=1$. If $\epsilon_{\rho}=\nu_2(\rho)\in\{1,-1\}$, then
\begin{align}
p_1&=-2s,\label{r8:eq:p1}\\
p_2&=-2+(4\epsilon_{\rho}-12)s,\label{r8:eq:p2}\\
p_4&=-2+4s\bigl(\nu_4(\rho)-\nu_4(t)-2\bigr).
\label{r8:eq:p4}
\end{align}
\end{proposition}
\begin{proof}
In the difference between the two indicator sums
\eqref{r8:eq:indicator} for $\rho$ and $t$, the $V$-contributions cancel.
Each $W$ contributes $c_W^2d\theta_W^n$.
If $U$ exists, its contribution is $-2\theta_U^n$, whereas that of
$U\otimes A$ is $(2+2d)\theta_U^n$ by Lemma~\ref{r8:lem:U}; their sum is
$2d\theta_U^n$. Thus the full residual contribution is $dp_n$.

Let $K_n$ denote the contribution from Table~\ref{r8:tab:fifteen} to this
difference. The table gives
\begin{equation}\label{r8:eq:Kn}
K_n=(12+7d)(1+(-1)^n)+8+4d
=\begin{cases}
4(d+2),&n\text{ odd},\\
18d+32,&n\text{ even}.
\end{cases}
\end{equation}
Therefore
\begin{equation}\label{r8:eq:all-moments}
p_n=\frac{D(\nu_n(\rho)-\nu_n(t))-K_n}{d}.
\end{equation}
The first indicators of $t,\rho$ vanish, so
$p_1=-4(d+2)/d=-2s$.

Exactly four objects in Table~\ref{r8:tab:fifteen} contain $t$, each once
and each of dimension $2(d+2)$. Their twists are $1,1,-1,-1$.
Their contribution to the twist-square trace of $I(t)$ is therefore
$8(d+2)=D$. The remaining summands of $I(t)$ also have total dimension
$D$, since $\dim I(t)=2D$. If $\nu_2(t)=-1$, their twist-square trace
would have to equal $-2D$, contradicting the triangle inequality.
Hence $\nu_2(t)=1$.

Finally $D/d=4s$ and $(18d+32)/d=2+8s$.
Substitution in \eqref{r8:eq:all-moments} for $n=2,4$ proves
\eqref{r8:eq:p2} and \eqref{r8:eq:p4}.
\end{proof}

Let $K=\QQ(\sqrt6)$. Fix a cyclotomic field $L=\QQ(\zeta_N)$ with
$24\mid N$ and with no prime divisors of $N$ other than $2,3$.
Here $\zeta_m=\exp(2\pi\ii/m)$. Define
\begin{equation}\label{r8:eq:trace-functions}
P=\frac1{[L:K]}\Tr_{L/K},\qquad
P(z)=a(z)+b(z)\sqrt6,\qquad a(z),b(z)\in\QQ.
\end{equation}
The functions $a,b$ are $\QQ$-linear. By transitivity of normalized
trace, their values on a fixed algebraic number do not change if $L$
is enlarged within this class of cyclotomic fields.

\begin{lemma}\label{r8:lem:root-table}
If $\zeta\in L$ has exact order $m$, the quantities needed below are
\begin{equation}\label{r8:eq:root-table}
\begin{array}{c|ccc}
m&|b(\zeta)|&|b(\zeta^2)|&a(\zeta^2)+a(\zeta^4)\\ \hline
1,2&0&0&2\\
3,6,8&0&0&-1\\
24&1/4&0&1/2\\
48&0&1/4&0\\
\text{all other }m&0&0&0.
\end{array}
\end{equation}
\end{lemma}
\begin{proof}
The conductor of $K$ is $24$. If $24\nmid m$, then
$K\cap\QQ(\zeta_m)=\QQ$, and $P(\zeta)$ is rational. If $24\mid m$
but $m>24$, choose $p\in\{2,3\}$ with $24\mid m/p$. The conjugates of
a primitive $m$th root over $\QQ(\zeta_{m/p})$ are obtained by
multiplication by the $p$th roots of unity, so their sum is zero.
Transitivity gives $P(\zeta)=0$.

For $m=24$, the kernel of the quadratic character associated with $K$
is $\{1,5,19,23\}\subset(\ZZ/24\ZZ)^\times$, and
\[
\zeta_{24}+\zeta_{24}^{5}+\zeta_{24}^{19}+\zeta_{24}^{23}=\sqrt6.
\]
The sum over the other coset is $-\sqrt6$. Hence a primitive
$24$th root has relative trace $\sqrt6/4$ or $-\sqrt6/4$.
This proves the first column of absolute values. Applying it to
$\zeta^2$ proves the second: the order of $\zeta^2$ is $24$ precisely
when $m=48$.

Since $a(z)$ is normalized absolute trace from $L$ to $\QQ$, the
usual primitive-root sum, or M\"obius inversion, gives
\begin{equation}\label{r8:eq:ramanujan}
a(\zeta^j)=\frac{\mu(e_j)}{\varphi(e_j)},\qquad
 e_j=\frac m{\gcd(m,j)}.
\end{equation}
Here $\mu$ is the M\"obius function and $\varphi$ is Euler's totient.
For $m=2^u3^v$, substitution of $j=2,4$ into \eqref{r8:eq:ramanujan}
gives the final column of \eqref{r8:eq:root-table}. In particular, if
$v\ge2$ both terms vanish, and if $v\le1$ the remaining values depend
only on $u\le3$, except that all larger $u$ give zero.
\end{proof}

\begin{theorem}\label{r8:thm:mixed-moment}
Let $\Lambda$ be a finite multiset of roots of unity whose orders have
no prime divisors other than $2,3$, and set
$p_j=\sum_{\zeta\in\Lambda}\zeta^j$. Choose $L$ containing these roots
and $\zeta_{24}$, and define $a,b$ by \eqref{r8:eq:trace-functions}. Then
\begin{equation}\label{r8:eq:mixed-moment}
|\Lambda|\ge 6|b(p_1)|+4|b(p_2)|-a(p_2+p_4).
\end{equation}
\end{theorem}
\begin{proof}
Each row of \eqref{r8:eq:root-table} gives the pointwise inequality
\[
1+a(\zeta^2)+a(\zeta^4)
\ge6|b(\zeta)|+4|b(\zeta^2)|.
\]
Sum over the multiset and apply the triangle inequality separately to
the two sums of $b$-coefficients. Linearity of $a,b$ gives
\eqref{r8:eq:mixed-moment}.
\end{proof}

\begin{example}\label{r8:ex:sharpness}
The conditions
\begin{equation}\label{r8:eq:sharp-moments}
p_1=-2\sqrt6,\qquad p_2=-2-8\sqrt6,\qquad p_4=-2
\end{equation}
are attained by a multiset of size $48$. To see this, put
\[
E_{24}=\{\zeta_{24}^r:r\in\{7,11,13,17\}\},\qquad
E_{48}=\{\zeta_{48}^r:0\le r<48,\ r\bmod24\in\{7,11,13,17\}\},
\]
and let $\mathcal P_m$ be the set of primitive $m$th roots. With an integer
multiple of a set denoting repeated copies, take
\[
\Lambda_*=2E_{24}\sqcup4E_{48}\sqcup\mathcal P_8
\sqcup\mathcal P_3\sqcup\mathcal P_6.
\]
The contributions to $(p_1,p_2,p_4)$ are respectively
\[
(-2\sqrt6,0,4),\quad (0,-8\sqrt6,0),\quad
(0,0,-4),\quad(-1,-1,-1),\quad(1,-1,-1).
\]
They sum to \eqref{r8:eq:sharp-moments}, and the cardinality is
$8+32+4+2+2=48$. Thus the numerical lower bound used below is exact at
the level of these three moments. No categorical realization of this
multiset is asserted.
\end{example}

\begin{theorem}\label{r8:nonexistence}
The near-Tambara--Yamagami fusion
category $\C=\C(\D,\kappa)$ with $n=2$ and $\kappa=8$ does not exist.
\end{theorem}
\begin{proof}
Work with the hypothetical pseudo-unitary category and the
multiset $\Lambda$ constructed before Proposition \ref{r8:prop:moments}.

The norm of the global dimension is
\[
N_{K/\QQ}(D)=(48+16\sqrt6)(48-16\sqrt6)=768=2^8\cdot3.
\]
Let $N_0$ be the Frobenius--Schur exponent of $\cC$. The Cauchy theorem
\cite[Theorem~3.9]{BNRW} identifies the prime ideals dividing the
global dimension with those dividing $N_0$ in the appropriate
cyclotomic integer ring. Every rational prime dividing $N_0$ must
therefore be $2$ or $3$. The exponent $N_0$ is the order of the twist
of $\cZ$ \cite[Theorem~5.5]{NS}, so all roots in $\Lambda$ have orders
supported on these two primes. We may apply
Theorem~\ref{r8:thm:mixed-moment}, taking $N=\operatorname{lcm}(24,N_0)$.

Equations \eqref{r8:eq:p1} and \eqref{r8:eq:p2} give
\[
b(p_1)=-2,\qquad b(p_2)=4\epsilon_{\rho}-12,\qquad a(p_2)=-2,
\qquad \epsilon_{\rho}=\nu_2(\rho)\in\{1,-1\}.
\]
In particular $|b(p_2)|\ge8$. By \eqref{r8:eq:fourth-cyclotomic},
$\nu_4(\rho)-\nu_4(t)\in\ZZ[\ii]$. Complex conjugation fixes $K$,
so its normalized relative trace to $K$ is rational. In
\eqref{r8:eq:p4}, the term multiplied by $\sqrt6$ consequently contributes
no rational coefficient. Thus $a(p_4)=-2$ and
\[
a(p_2+p_4)=-4.
\]
The mixed-moment inequality now requires
\[
46=|\Lambda|\ge6\cdot2+4\cdot8+4=48,
\]
a contradiction. This proves the theorem.
\end{proof}

\section{The rational near-Tambara--Yamagami categories}\label{sec:rational}

Assume in this section that $d=d_+\in\Z$. Then $\C$ is integral and
pseudo-unitary. The following argument is independent of the irrational
case and, in particular, does not assume $n=2$ at the outset.

\subsection{Integral parameter restrictions}
\begin{lemma}\label{lem:integral-param}
There is a positive integer $m$ such that
\[
d_-= -m,\qquad dm=2n^2,\qquad d>m,\qquad \kappa=d-m.
\]
Moreover, there are integers $a,m\geq1$ with
\begin{equation}\label{eq:am-param}
n=am,\qquad d=2a^2m,\qquad \kappa=m(2a^2-1).
\end{equation}
\end{lemma}
\begin{proof}
The equation $d(d-\kappa)=2n^2$ shows that the other root is
$d_-=-m$ for a positive integer $m$. The strict inequality $d>m$
follows from $\kappa>0$; the case $\kappa=0$ would give
$d=n\sqrt2$. The formal-codegree summand corresponding to $d_-$
has dimension
\[
\frac{N+d^2}{N+d_-^2}=-\frac d{d_-}=\frac{d^2}{N}.
\]
This rational algebraic integer is an integer. Thus $(d/n)^2=2r$
for some $r\in\Z_{>0}$. A rational number whose square is an
integer is an integer; since its square is even, $d/n$ is even.
Write $d/n=2a$ with $a\geq1$. The identities $d=2an$ and $dm=2n^2$ give $n=am$.
Substitution gives $d=2a^2m$ and
$\kappa=d-m=m(2a^2-1)$, as required.
\end{proof}

\begin{lemma}\label{lem:integral-unit}
With the notation of \eqref{eq:am-param}, write the multiplicity-free
decomposition
\[
 I(\one)\cong \one\oplus A\oplus B\oplus
 \bigoplus_{i=1}^{n^2-1}Y_i .
\]
Here $\one,A,B,Y_1,\ldots,Y_{n^2-1}$ are pairwise nonisomorphic simple
objects of $\ZC$, each occurring once, and
\[
 \FPdim(A)=2a^2,\qquad
 \FPdim(B)=2a^2+1,\qquad
 \FPdim(Y_i)=4a^2+2.
\]
Every summand $Z$ satisfies $[F(Z),\one]=1$.
\end{lemma}
\begin{proof}
This is Lemma~\ref{lem:codegrees} after substituting
$d=2a^2m$ and $m=-d_-$. The formal-codegree theorem gives
multiplicity one because $K_0(\C)$ is commutative, and adjunction gives
the unit multiplicity.
\end{proof}

\begin{proposition}\label{prop:m-at-most-two}
In \eqref{eq:am-param}, $m\in\{1,2\}$.
\end{proposition}
\begin{proof}
The largest dimension of a summand of $I(\one)$ is $4a^2+2$ by
Lemma~\ref{lem:integral-unit}. Since $\kappa>0$, equation
\eqref{eq:FI-general} implies that some summand $Z$ has
$[F(Z),\rho]>0$. It also contains $\one$, and therefore
\[
\FPdim(Z)\geq 1+d=1+2a^2m.
\]
If $m\geq3$, then $1+2a^2m>4a^2+2$ for every $a\geq1$.
This is a contradiction. Thus $m\leq2$.
\end{proof}

\subsubsection{The branch $m=1$}

Here $n=a$, $d=2n^2$, and $\kappa=2n^2-1$. Put $N=2n^2=d$.
The summands of $I(\one)$ have dimensions $1,N,N+1,2N+2$.

\begin{lemma}\label{lem:m1-images}
Let $A,B,Y_i$ be the simple summands named in
Lemma~\ref{lem:integral-unit}. Thus
\[
 \FPdim(A)=N,\qquad \FPdim(B)=N+1,\qquad
 \FPdim(Y_i)=2N+2 .
\]
After relabeling $G$ and indexing the $Y_i$ by
$g\in G\setminus\{\one\}$ as $X_g$, their forgetful images are
\[
F(A)\cong \bigoplus_{g\in G}g\oplus n t,\qquad F(B)\cong \one\oplus \rho,\qquad
F(X_g)\cong \one\oplus g\oplus 2 \rho\quad(g\in G\setminus\{\one\}).
\]
\end{lemma}
\begin{proof}
By Lemma~\ref{lem:integral-unit}, every image contains $\one$ exactly once.
The dimensions of $A,B,Y_i$ and the fact that the total $\rho$
coefficient in $F(I(\one))$ is $\kappa=N-1$ give the upper bounds
$0,1,2$ for the $\rho$ multiplicities in $A,B,Y_i$, respectively.
Their sum is at most
\[
 0+1+2(n^2-1)=N-1,
\]
which is the total coefficient. Thus all three bounds are attained.
The remaining dimension gives $F(B)\cong \one\oplus \rho$ and shows that each
$Y_i$ has the form $\one\oplus u_i\oplus 2 \rho$ with $u_i\in G$.
The total $t$ coefficient is $n$, so all copies of $t$ occur in $F(A)$.
Write
\[
 F(A)\cong \bigoplus_{g\in G}a_g g\oplus n t .
\]
Since $\FPdim(A)=N=2n^2$, dimension gives
$\sum_{g\in G}a_g=n^2$. Using adjunction for $g$ and $t$,
\[
 n=[F(I(g)),t]
   =\sum_Z[F(Z),g][F(Z),t]\geq a_g n,
\]
so $a_g\leq1$ for every $g$. Hence $a_g=1$ for all $g\in G$.
Finally, the coefficient of every nonidentity $g$ in $F(I(\one))$ is two,
whereas $F(A)$ contributes one. Thus the $Y_i$ are indexed exactly once
by the nonidentity elements of $G$, giving the notation $X_g$ and the
displayed formulas.
\end{proof}

We use the convention for the image of a tensor functor from
Section~\ref{general-NTYC}.

\begin{lemma}\label{lem:ratio}
Let $L:\mathcal U\to\mathcal V$ be a dominant tensor functor with
right adjoint $R$. Then
\begin{equation}\label{eq:ratio}
\frac{\FPdim(\mathcal U)}{\FPdim(\mathcal V)}
=\FPdim(R(\one))=\sum_{X\in\Irr(\mathcal U)}[L(X),\one]\FPdim(X).
\end{equation}
If $R(\one)\cong \one$, then $L$ is fully faithful. For a fusion subcategory
$\mathcal U\subset\ZC$ and $L=F|_{\mathcal U}$, $R(\one)$ is precisely
the part of $I(\one)$ supported in $\mathcal U$.
\end{lemma}
\begin{proof}
The image of the canonical regular element of $K_0(\mathcal U)$ is
$\FPdim(\mathcal U)/\FPdim(\mathcal V)$ times the regular element
of $K_0(\mathcal V)$ \cite[Proposition~3.3.13]{EGNO}.
Taking the coefficient of $\one$ and using adjunction proves
\eqref{eq:ratio}. If $R(\one)\cong \one$, rigidity and adjunction identify
$\operatorname{Hom}(L(X),L(Y))$ with $\operatorname{Hom}(X,Y)$.
The final assertion follows from $[R(\one),Z]=[F(Z),\one]=[I(\one),Z]$.
\end{proof}

\begin{proposition}\label{prop:m1}
If $m=1$, then $n=2$, $G\cong C_2\times C_2$, $d=8$, and
$\kappa=7$.
\end{proposition}
\begin{proof}
Let $\C_{\mathrm{pt}}$ be the pointed subcategory of $\C$ and set
$\mathcal P=F^{-1}(\C_{\mathrm{pt}})\subset\ZC$. By
Lemma~\ref{lem:m1-images}, the only summand of $I(\one)$ whose
forgetful image is pointed is $\one$. Thus the right adjoint of
$F|_{\mathcal P}$ sends $\one$ to $\one$, and
Lemma~\ref{lem:ratio} makes $F|_{\mathcal P}$ fully faithful.
In particular, the forgetful image of a simple object of $\mathcal P$
is a single invertible object, and any such invertible has at most
one central lift in $\mathcal P$.

The unique dimension-$N$ summand $A$ of $I(\one)$ is self-dual.
Its forgetful image is the regular object of $\D$, so
$F(A)^2\cong N F(A)$. Among the summands of $I(\one)$, only
$\one$ and $A$ can occur in $A^2$, since every other summand has a
$\rho$ constituent after applying $F$. Adjunction and
$[F(A^2),\one]=N$ therefore give
\begin{equation}\label{eq:m1-A-square-initial}
 A^2\cong\one\oplus(N-1)A\oplus W,
 \qquad F(W)\cong\bigoplus_{g\ne\one}g\oplus nt.
\end{equation}
For a central simple $Y$, put $a_g(Y)=[F(Y),g]$ and
$b(Y)=[F(Y),t]$. The induction-adjunction formula gives
\[
 n=[I(g),I(t)]=\sum_Y a_g(Y)b(Y)\qquad(g\in G).
\]
The summand $A$ contributes $1\cdot n=n$, so nonnegativity implies
\begin{equation}\label{eq:m1-no-mixed}
 a_g(Y)b(Y)=0\qquad(g\in G,\;Y\ne A).
\end{equation}
Every simple constituent of $W$ is consequently supported, after
applying $F$, either on $G$ or on $t$ alone. Full faithfulness on
$\mathcal P$ and the coefficients of $F(W)$ show that the pointed
constituents, together with $\one$, give unique central lifts
$\widetilde g$ of all $g\in G$.

Let $T$ be a simple constituent of $W$ with $F(T)\cong bt$,
where $b>0$. The object $T\otimes T^*$ is in $\mathcal P$, and
\[
 F(T\otimes T^*)\cong\bigoplus_{g\in G}b^2g,
 \qquad
 [T\otimes T^*,\widetilde g]
 =[T,\widetilde g\otimes T]\le1.
\]
Uniqueness of the lifts therefore forces $b=1$. If $T'$ is another
simple with $F(T')\cong t$, then $T\otimes T'^*$ lies in
$\mathcal P$ and its forgetful image contains $\one$ once. Hence
$[T,T']=[T\otimes T'^*,\one]=1$, so $T\cong T'$.
It follows that $T\cong T^*$, $\widetilde gT\cong T$, and
\[
 T^2\cong\bigoplus_{g\in G}\widetilde g,
 \qquad
 \widetilde g\widetilde h\cong\widetilde{gh}.
\]
The coefficient of $t$ in $F(W)$ gives $[W,T]=n$; reciprocity
and dimensions now yield
\[
 [T A,A]=[A^2,T]=n,\qquad TA\cong nA.
\]
Thus the subcategory on $\{\widetilde g:g\in G\}\cup\{T\}$ is a
braided Tambara--Yamagami category. By
\cite[Theorem~1.2(1)]{Siehler}, $G$ is elementary abelian of
exponent two. Since $|G|=n^2$ and $n\ge2$, write $n=2^s$ with $s\ge1$.

We next establish the closure and twists needed for the Tannakian
argument. The simple $\widetilde gA$ has dimension $N$, and
$F(\widetilde gA)$ contains $\one$. It is therefore a summand of
$I(\one)$, whose unique dimension-$N$ summand is $A$. Hence
$\widetilde gA\cong A$. Equation~\eqref{eq:m1-A-square-initial}
becomes
\begin{equation}\label{eq:m1-A-square}
 A^2\cong\one\oplus(N-1)A\oplus
 \bigoplus_{g\ne\one}\widetilde g\oplus nT.
\end{equation}
Together with the preceding product formulas and self-duality of
$A,T$, this proves that
\[
 \{\widetilde g:g\in G\}\cup\{T,A\}
\]
is the complete simple-object list of a fusion subcategory
$\mathcal E\subset\ZC$.

The character associated with $A$ is $\chi_-$, with values
$\chi_-(g)=1$, $\chi_-(t)=n$, and $\chi_-(\rho)=-1$.
The induction-character identity \eqref{eq:central-character} gives
\[
 \Ss_{A,A}
 =\dim(A)\chi_-(F(A))=N^2.
\]
As $A\subset I(\one)$, its twist is $\theta_A=1$.
Balancing applied to \eqref{eq:m1-A-square} consequently gives
\[
 N^2=1+(N-1)N+
 \sum_{g\ne\one}\theta_{\widetilde g}+n^2\theta_T,
\]
or
\[
 \sum_{g\ne\one}\theta_{\widetilde g}+n^2\theta_T
 =N-1=(n^2-1)+n^2.
\]
All twists have absolute value one. Equality with the sum of the
positive weights forces $\theta_{\widetilde g}=	\theta_T=1$.
Thus every simple of $\mathcal E$ has trivial twist. Balancing
implies that $\mathcal E$ is symmetric. Its inherited positive
spherical structure, together with the trivial twist, is the
canonical symmetric structure; hence \cite[Theorem~9.9.22]{EGNO}
makes $\mathcal E$ Tannakian.

Write $\mathcal E\simeq\Rep(K)$. The complete simple list gives
\[
 |K|=|G|+n^2+N^2=N(N+1).
\]
Its Tambara--Yamagami subcategory is $\Rep(K/V)$ for a normal
subgroup $V\lhd K$, with $|K/V|=N$ and $|V|=N+1$.
The only irreducible representation not inflated from $K/V$ is
$A$, of dimension $N$. Restriction of the regular representation gives
\[
 \operatorname{Res}^K_V\CC[K]\cong N\CC[V].
\]
The inflated part of $\CC[K]$ restricts to $N\one_V$, while $A$
occurs with multiplicity $N$. Cancelling that multiplicity yields
\[
 \operatorname{Res}^K_VA\oplus\one_V\cong\CC[V].
\]
Consequently $\chi_A(v)=-1$ for every $v\in V\setminus\{1\}$.
An irreducible representation inflated from $K/V$ has character
value equal to its dimension on all of $V$ (in particular $T$
has value $n$). Thus every irreducible character of $K$ is constant
on $V\setminus\{1\}$, so character separation makes that set a
single $K$-conjugacy class.

All its elements have a common prime order $p$: if the common
order were composite, a nonidentity proper power would have a
different order. Cauchy's theorem shows that $V$ is a $p$-group.
Its nontrivial center $Z(V)$ is characteristic in $V$, and hence
normal in $K$. Transitivity on $V\setminus\{1\}$ forces
$Z(V)=V$. Therefore $V$ is abelian of exponent $p$, so
$V\cong(C_p)^r$ and $N+1=p^r$.

We obtain $2^{2s+1}+1=p^r$. Reduction modulo three gives $p=3$;
reduction modulo eight gives $r=2u$. Hence
\[
 (3^u-1)(3^u+1)=2^{2s+1}.
\]
The two factors are positive powers of two differing by two, so
they are $2$ and $4$. Therefore $u=s=1$ and $n=2$.
The group $G$ has order four and exponent two, so
$G\cong C_2\times C_2$, and $d=N=8$, $\kappa=N-1=7$.
\end{proof}

\subsubsection{The branch $m=2$}

Write $n=2a$ and $q=n^2/2=2a^2$. Then
\begin{equation}\label{eq:m2-params}
d=2q,\qquad \kappa=2q-2,\qquad
\FPdim(\D)=4q,\qquad \Delta=4q(q+1).
\end{equation}

\begin{lemma}\label{lem:m2-distinguished}
Let $A,B,Y_i$ denote the objects in
Lemma~\ref{lem:integral-unit}. In the $m=2$ branch,
\[
 \FPdim(A)=q,\qquad \FPdim(B)=q+1,\qquad
 \FPdim(Y_i)=2q+2.
\]
Exactly one of the $2q-1$ summands $Y_i$ of dimension $2q+2$
has a forgetful image without
$\rho$. Denote this distinguished central object by $X_0$. Then
$F(X_0)$ contains $t$.
\end{lemma}
\begin{proof}
The summands of dimensions $q$ and $q+1$ cannot contain $\rho$,
and each $2q+2$ summand contains at most one copy of $\rho$.
Since the total $\rho$ multiplicity is $2q-2$, exactly one, $X_0$,
contains no $\rho$. For the other summands, dimension gives
$F(X)\cong \one\oplus g\oplus \rho$ with $g\ne \one$.
Writing $q=2a^2$ and $n=2a$, the unit multiplicity bounds give
\[
[F(A),t]\le a-1,\qquad [F(B),t]\le a.
\]
Since $[F(I(\one)),t]=2a$, one has $[F(X_0),t]\ge1$.
\end{proof}

\begin{proposition}\label{prop:m2-centralizer}
Let $\mathcal E_2$ be the full fusion subcategory of $\ZC$ generated by
the simple objects whose forgetful images lie in $\D$:
\[
 \mathcal E_2=F^{-1}(\D).
\]
The prime denotes the M\"uger centralizer inside the braided category
$\ZC$; explicitly,
\[
 \mathcal E_2'
 =\left\{Z\in\ZC:
 c_{W,Z}\circ c_{Z,W}=\operatorname{id}_{Z\otimes W}
 \text{ for every }W\in\mathcal E_2\right\},
\]
where $c$ is the braiding of $\ZC$. Put
$\mathcal J=\mathcal E_2'$. Then
\[
\FPdim(\mathcal E_2)=4\Delta,\qquad
\FPdim(\mathcal J)=\Delta/4=q(q+1),
\]
and $\mathcal J$ is Tannakian. Moreover,
\[
\mathcal J\simeq\Rep(K),\quad |K|=q(q+1),\qquad
F(\mathcal J)\simeq\Rep(H),\quad |H|=q,
\]
where $H\le K$ is abelian, $[K:H]=q+1$, and its character
group $\widehat H$ identifies with a subgroup of $G$. Under these
equivalences, $F|_{\mathcal J}$ is the restriction functor
$\operatorname{Res}^K_H$, and its right-adjoint unit algebra is
$\operatorname{Fun}(K/H)\cong\one\oplus A$.
\end{proposition}
\begin{proof}
The image of $F|_{\mathcal E_2}$ contains $t$ by Lemma~\ref{lem:m2-distinguished}
and hence is all of $\D$. The part of $I(\one)$ in $\mathcal E_2$ is
$\one\oplus A\oplus B\oplus X_0$, of dimension $4(q+1)$. Lemma~\ref{lem:ratio}
gives $\FPdim(\mathcal E_2)=4q\cdot4(q+1)=4\Delta$.
The M\"uger dimension formula in the nondegenerate category $\ZC$
\cite[Corollary~8.21.7]{EGNO} gives
$\FPdim(\mathcal J)=\Delta^2/(4\Delta)=\Delta/4$.

If $F(\mathcal J)$ contained $\rho$, it would be all of $\C$ and
the ratio in Lemma~\ref{lem:ratio} would be $1/4$, impossible because
the right adjoint unit image has integral positive dimension. If it
contained $t$, it would be all of $\D$, and the ratio would be
$(q+1)/4$, not an integer because $q$ is even. Hence $F(\mathcal J)$
is pointed. In particular $\mathcal J\subseteq\mathcal E_2$,
while the double-centralizer theorem gives
$\mathcal J'=\mathcal E_2$. Hence $\mathcal J\subseteq\mathcal J'$
and $\mathcal J$ is symmetric.

If $\mathcal J$ were not Tannakian, its maximal Tannakian subcategory
would have half the dimension by \cite[Corollary~9.9.32]{EGNO}.
Its centralizer would then have dimension $8\Delta$. This centralizer
contains $\mathcal E_2$; since its dimension is larger than
$\FPdim(\mathcal E_2)=4\Delta$, it cannot be contained in
$\mathcal E_2$, so its forgetful image contains $\rho$ and is all of
$\C$. The ratio in Lemma~\ref{lem:ratio} is therefore $8$.
On the other hand it contains $\one,A,B,X_0$, so its right adjoint unit image
has dimension at least $4(q+1)\ge12$, a contradiction. Hence
$\mathcal J\simeq\Rep(K)$.

It remains to identify the right-adjoint unit image exactly. Let
$\eta_B$ be the character of $K_0(\C)$ corresponding to the summand
$B$; explicitly,
\[
 \eta_B(g)=1\quad(g\in G),\qquad
 \eta_B(t)=-n,\qquad \eta_B(\rho)=0,
 \qquad \FPdim(B)=q+1.
\]
Equation~\eqref{eq:central-character} gives, for every central simple $Z$,
\[
 \Ss_{B,Z}=(q+1)\eta_B(FZ),
\]
with the normalization of \eqref{eq:central-character}. Since
$F(X_0)$ contains $t$ and no $\rho$, write
$F(X_0)\cong \bigoplus_g a_g g\oplus \gamma t$ with $\gamma\ge1$. If $B$ belonged
to $\mathcal J$, then $B$ would centralize $X_0$, and hence
\[
 (q+1)\eta_B(FX_0)=\Ss_{B,X_0}
   =\FPdim(B)\FPdim(X_0).
\]
But
\[
 \eta_B(FX_0)=\sum_g a_g-n\gamma
   <\sum_g a_g+n\gamma=\FPdim(X_0),
\]
a contradiction. Thus $B\notin\mathcal J$. The object $X_0$ is not in
$\mathcal J$ because its image contains $t$, and every other $Y_i$ is not
in $\mathcal J$ because its image contains $\rho$. Therefore the part of
$I(\one)$ supported in $\mathcal J$ is contained in $\one\oplus A$. It cannot
be only $\one$: in that case the pointed image would have dimension
$q(q+1)>|G|=2q$, contradicting that the pointed subcategory of $\C$ has
exactly $|G|$ simples. Hence the right-adjoint unit image is exactly
$\one\oplus A$. The ratio in Lemma~\ref{lem:ratio} gives
$\FPdim(F(\mathcal J))=q$.

Set $\mathcal B=F(\mathcal J)$ and $L=F|_{\mathcal J}:\mathcal J\to
\mathcal B$. Restricting the given half-braidings to objects of
$\mathcal B$ makes $L$ a dominant central functor. By
\cite[Lemma~3.5]{DMNO}, its right-adjoint unit object
$A_0=\one\oplus A$ is a connected \'etale algebra in
$\mathcal J\simeq\Rep(K)$, and
\[
 \mathcal B\simeq\Rep(K)_{A_0}.
\]
A finite-dimensional commutative separable algebra in $\Rep(K)$
is the function algebra of a finite $K$-set. Connectedness means
that the set is transitive. Thus, for a stabilizer subgroup $H\le K$,
\[
 A_0\cong\operatorname{Fun}(K/H),\qquad [K:H]=\dim(A_0)=q+1.
\]
Taking the fiber of an equivariant module at the coset $H$ gives
$\Rep(K)_{A_0}\simeq\Rep(H)$; the free-module functor corresponds
to $\operatorname{Res}^K_H$. It follows that $L$ is this restriction
functor under the stated equivalences. Since $\mathcal B$ is pointed,
$H$ is abelian, and $|H|=|K|/(q+1)=q$. Its simple representations
are its characters; their images in $\C_{\mathrm{pt}}$ identify
$\widehat H$, rather than the stabilizer $H$ itself, with a
subgroup of $G$.
\end{proof}

\begin{proposition}\label{prop:m2}
If $m=2$, then $n=2$, $G\cong C_2\times C_2$, $d=4$, and $\kappa=2$.
\end{proposition}
\begin{proof}
By Proposition~\ref{prop:m2-centralizer}, $K$ acts transitively on
$\Omega=K/H$, where $|\Omega|=q+1$, $|H|=q$, and
\[
 \CC[\Omega]\cong\one\oplus A.
\]
Thus $\dim\End_K(\CC[\Omega])=2$. This dimension counts the
$K$-orbits on $\Omega\times\Omega$; the diagonal is one orbit,
so the action is doubly transitive. Moreover,
\[
 |K|=q(q+1)=|\Omega|(|\Omega|-1).
\]
The stabilizer of an ordered pair of distinct points is therefore
trivial. In particular the action is faithful and sharply doubly
transitive.

The Frobenius kernel theorem for finite sharply doubly transitive
groups \cite[Section~7.6]{DixonMortimer} gives
\[
 K=V\rtimes H,\qquad |V|=q+1,
\]
where $V$ is a regular normal subgroup and $H$ acts regularly by
conjugation on $V\setminus\{1\}$. For completeness, this
transitivity also shows that $V$ is elementary abelian. All its
nonidentity elements have a common prime order $p$; Cauchy's
theorem makes $V$ a $p$-group. Since $Z(V)$ is nontrivial and
characteristic, transitivity forces $Z(V)=V$. Thus
$V\cong(C_p)^r$, $q+1=p^r$, and $p\nmid|H|$.

The group $H$ is abelian by Proposition~\ref{prop:m2-centralizer}.
Its representation on
$V\otimes_{\mathbb F_p}\overline{\mathbb F}_p$ is therefore a
sum of one-dimensional characters. Every occurring character is
faithful. Indeed, if a nonidentity $h\in H$ lay in its kernel,
then $h-I$ would have a nonzero kernel over
$\overline{\mathbb F}_p$. Its determinant, computed over
$\mathbb F_p$, would vanish, so it would already have a nonzero
kernel over $\mathbb F_p$. This contradicts the fixed-point-free
action of $h$ on $V\setminus\{0\}$.
A faithful character embeds $H$ in the multiplicative group of
some finite extension of $\mathbb F_p$, which is cyclic. Hence
$H$ and its character group $\widehat H$ are cyclic of order $q$.
Since $\widehat H\le G$, the group $G$ has an element of order
$q=n^2/2$.

Lemma~\ref{lem:order-bound} gives $\ord(g)\le n$ for every
$g\in G$. Hence $n^2/2\le n$, and $n\ge2$ implies $n=2$.
Now $G$ has order four and every element has order at most two,
so $G\cong C_2\times C_2$. Equation~\eqref{eq:m2-params} gives
$d=4$ and $\kappa=2$.
\end{proof}

\begin{theorem}\label{thm:integral}
If $d_+\in\Z$, then necessarily
\[
(n,d_+,\kappa)=(2,4,2)\quad\text{or}\quad(2,8,7),
\]
and in either case $G\cong C_2\times C_2$ and $\operatorname{rank}(\C)=6$.

\end{theorem}
\begin{proof}
Proposition~\ref{prop:m-at-most-two} leaves $m=1$ or $m=2$. Propositions
\ref{prop:m1} and \ref{prop:m2} give the two alternatives.
\end{proof}

\medbreak
Thus the integral case leaves $R_2$ and $R_7$, as defined in
\eqref{eq:Rkappa}. Example~\ref{example:PSU} categorifies $R_7$;
we now construct a categorification of $R_2$.

\subsection{A group-theoretical categorification of \texorpdfstring{$R_2$}{R2}}\label{r2:sec:positive}

\subsubsection{Group-theoretical data}

For a finite group $L$, a normalized unitary $3$-cocycle $\omega$, a
subgroup $H$, and a normalized unitary $2$-cochain $\psi$ satisfying
$\delta\psi=\omega|_H$, denote by $\cC(L,\omega,H,\psi)$ the category of
bimodules over the twisted regular algebra of $H$ in $\Vect_L^\omega$.
Its simple objects are indexed by pairs $(x,\pi)$, where $x$ represents
a double coset $HxH$ and $\pi$ is an irreducible projective representation
of $H_x=H\cap xHx^{-1}$. Their dimensions are
\begin{equation}\label{r2:eq:double-coset-dimension}
\FPdim(x,\pi)=[H:H_x]\dim\pi.
\end{equation}
The double-coset support of tensor products is multiplicative, and the
component supported on $H$ is $\Rep(H)$; see
\cite[Theorems~4.1 and~5.1]{GN}.

Set
\[
L=\PGL_2(\FF_3)\cong S_4.
\]
For $x\in L$, choose the representative $s(x)\in\GL_2(\FF_3)$ whose first
nonzero entry, in row-major order, equals $1$. Define
\begin{equation}\label{r2:eq:alpha-beta}
s(x)s(y)=(-I)^{\beta(x,y)}s(xy),\qquad
\alpha(x)=
\begin{cases}
0,&\det s(x)=1,\\
1,&\det s(x)=2.
\end{cases}
\end{equation}
Here $\beta$ takes values in $\FF_2$ and is the normalized $2$-cocycle
of the central extension by $\{I,-I\}$; $\alpha:L\to\FF_2$ is a
homomorphism. The cup product gives a normalized unitary $3$-cocycle
\begin{equation}\label{r2:eq:omega}
\omega(x,y,z)=(-1)^{\alpha(x)\beta(y,z)}.
\end{equation}

Take $H=\{1,a,b,c\}$, with representatives
\begin{equation}\label{r2:eq:H}
s(1)=I,\qquad
s(a)=\begin{pmatrix}0&1\\1&0\end{pmatrix},\qquad
s(b)=\begin{pmatrix}0&1\\2&0\end{pmatrix},\qquad
s(c)=\begin{pmatrix}1&0\\0&2\end{pmatrix}.
\end{equation}
In the action on $\mathbb P^1(\FF_3)$, this is the subgroup generated by
two disjoint transpositions. In particular, $H\cong C_2^2$ is not normal.
Define $\psi(x,y)=\ii^{q(x,y)}$, where $q(1,x)=q(x,1)=0$ and
\begin{equation}\label{r2:eq:q-table}
\begin{array}{c|rrr}
q&a&b&c\\ \hline
a&0&2&0\\
b&3&0&1\\
c&1&0&1
\end{array}
\qquad\text{in }\ZZ/4\ZZ.
\end{equation}

\begin{lemma}\label{r2:lem:cochains}
The cochains \eqref{r2:eq:omega} and \eqref{r2:eq:q-table} satisfy
$\delta\psi=\omega|_H$.
\end{lemma}
\begin{proof}
Write $x=c^v a^u$, $y=c^s a^r$, and $z=c^w a^p$, with all six exponents
in $\{0,1\}$. The representatives in \eqref{r2:eq:H} have
$s(c)s(a)=s(b)$ and $s(a)s(c)=-s(c)s(a)$, so
\[
\alpha(x)=u+v\pmod2,\qquad \beta(y,z)=rw\pmod2.
\]
The table \eqref{r2:eq:q-table} is equivalently
\[
q(x,y)=v(r+s)+2(vrs+urs+uvr)\pmod4.
\]
Expanding products in $H$ by the identity
$r\mathbin{\oplus}s=r+s-2rs$ for binary coordinates gives
\[
q(y,z)+q(x,yz)-q(xy,z)-q(x,y)
\equiv2(u+v)rw\pmod4.
\]
This is precisely the exponent identity for
$\delta\psi(x,y,z)=\omega(x,y,z)$.
\end{proof}

\subsubsection{The three double cosets}

Choose
\begin{equation}\label{r2:eq:double-representatives}
x=\left[\begin{pmatrix}1&1\\1&2\end{pmatrix}\right],\qquad
u=\left[\begin{pmatrix}0&1\\1&1\end{pmatrix}\right].
\end{equation}
Label the four projective points by $(1,0),(0,1),(1,1),(1,2)$.
Then $a=(12)$, $c=(34)$, $x=(13)(24)$, and $u=(1234)$.
The normalizer of $H=\langle(12),(34)\rangle$ has order $8$, and $x$
lies in its nontrivial coset modulo $H$. Moreover,
$uHu^{-1}=\langle(23),(14)\rangle$ has trivial intersection with $H$.
Thus the double cosets represented by $1,x,u$ have sizes $4,4,16$ and
exhaust $L$:
\begin{equation}\label{r2:eq:three-cosets}
L=H\sqcup xH\sqcup HuH,\qquad N_L(H)=H\sqcup xH.
\end{equation}
For any $y\in L$, the orbit--stabilizer formula for the left action of
$H\times H$ on $HyH$ is
\[
 |HyH|=\frac{|H|^2}{|H\cap yHy^{-1}|}.
\]
For $y=1$ and $y=x$, the intersection is $H$ because both elements lie
in $N_L(H)$, giving size $4$ in each case. For $y=u$, the displayed
conjugate subgroup has trivial intersection with $H$, giving size $16$.
The sizes add to $24=|L|$, so no further double coset is possible.

For clarity, in the conventions of \cite[Section~5.1]{GN}, the projective
factor on $H_x$ is
\begin{equation}\label{r2:eq:factor}
\psi^x(h,k)=\psi(h,k)\psi(x^{-1}k^{-1}x,x^{-1}h^{-1}x)
\frac{\omega(h,k,x)\omega(h,kx,x^{-1}k^{-1}x)}
{\omega(hkx,x^{-1}k^{-1}x,x^{-1}h^{-1}x)}.
\end{equation}
Substitution of \eqref{r2:eq:alpha-beta}--\eqref{r2:eq:q-table} gives the exponent
matrix for $\psi^x$, in the order $1,a,b,c$,
\begin{equation}\label{r2:eq:factor-table}
\begin{pmatrix}
0&0&0&0\\
0&1&1&0\\
0&3&0&1\\
0&2&3&1
\end{pmatrix}\pmod4.
\end{equation}
For example, substituting $(h,k)=(a,a),(a,b),(b,a),(b,c)$ into
\eqref{r2:eq:factor} gives the exponents $1,1,3,1$, respectively; the
other entries are obtained by the same four-element calculation and are
listed in the displayed matrix. Hence
$\psi^x(h,k)/\psi^x(k,h)=-1$ for any two distinct nonidentity elements
of $H$. Its radical is therefore trivial. Equivalently,
the ordinary algebraic center is one-dimensional:
\[
Z_{\mathrm{alg}}\bigl(\CC_{\psi^x}[H]\bigr)=\CC,\qquad
\CC_{\psi^x}[H]\cong M_2(\CC).
\]

\begin{theorem}\label{r2:prop:R2}
The category $\cC(L,\omega,H,\psi)$ defined above is a unitary
categorification of $R_2$.
\end{theorem}
\begin{proof}
The component over $H$ consists of four invertibles indexed by
$\widehat H\cong C_2^2$. The nondegenerate factor \eqref{r2:eq:factor-table}
gives one simple object $t$ of dimension $2$ over $xH$. The trivial
stabilizer of $HuH$ gives one simple object $\rho$ of dimension $4$.
These are all the simple objects, by \eqref{r2:eq:double-coset-dimension}.

Both $xH$ and $HuH$ are invariant under inversion: the first is the
nonidentity coset in the group $N_L(H)/H$ of order two, and the second is
the complement of $N_L(H)$ in $L$. Thus $t$ and $\rho$ are self-dual.
Tensoring by an invertible preserves each of these supports, so it fixes
$t$ and $\rho$ on both sides. The product of the support $xH$ with itself
is $H$. For every $g\in\widehat H$, Frobenius reciprocity and $g\otimes t\cong t$
give
\[
 [t\otimes t,g]=[t,g\otimes t]=1,
\]
while no other simple is supported on $H$; hence
$t\otimes t\cong \bigoplus_{g\in\widehat H}g$.
Multiplication by $N_L(H)$ preserves its complement $HuH$ on both sides.
Therefore $t\otimes\rho$ and $\rho\otimes t$ are supported on $HuH$, where $\rho$ is
the only simple object. Comparing dimensions gives
$t\otimes\rho\cong \rho\otimes t\cong 2 \rho$. Finally, reciprocity gives, for every
$g\in\widehat H$,
$[\rho\otimes \rho,g]=[\rho,g\otimes\rho]=1$, and
$[\rho\otimes \rho,t]=[\rho,t\otimes\rho]=2$. Thus
\[
\rho\otimes \rho\cong \bigoplus_{g\in\widehat H}g\oplus 2 t\oplus m \rho
\]
for some $m\ge0$. The dimension equation $16=4+4+4m$ yields $m=2$.

The associator in $\Vect_L^\omega$ is unitary. The twisted regular algebra
of $H$, with its orthonormal homogeneous basis and the unit-modulus
multiplication coefficients $\psi$, is the standard separable unitary
algebra. Its bimodule category, with relative tensor product, is a unitary
fusion category. This proves existence with the asserted structure.
\end{proof}

We can now prove the classification stated in the introduction.

\begin{proof}[Proof of Theorem~\ref{thm:classification}]
If $d_+\notin\Z$, Proposition~\ref{prop:rank-reduction} and
Theorem~\ref{thm:main} leave $n=2$ and $\kappa\in\{0,8\}$.
Theorem~\ref{r8:nonexistence} excludes $\kappa=8$, while
Theorem~\ref{thm:kappa-zero} treats $\kappa=0$ and supplies realizations.
If $d_+\in\Z$, Theorem~\ref{thm:integral} gives exactly the two
possibilities $(n,d_+,\kappa)=(2,4,2),(2,8,7)$.
Lemma~\ref{lem:order-bound} identifies $G$ with the Klein four-group,
and \eqref{eq:fusion} gives $K_0(\C)\cong R_\kappa$.
Theorem~\ref{r2:prop:R2} and Example~\ref{example:PSU} establish
existence for $R_2$ and $R_7$, respectively. The dimension table follows
from $\FPdim(\C)=8+d_+^2$.
\end{proof}

\subsection*{Acknowledgments}
We sincerely thank Andrew Schopieray for proposing this topic and helping us prove Proposition 3.1. S\'ebastien Palcoux is supported by the National Natural Science
Foundation of China (NSFC), Grant no.~12471031. GPT-6 Astra was used as a research assistant. The authors take full responsibility for the final manuscript.


\begin{thebibliography}{99}

\bibitem{BNRW}
P.~Bruillard, S.-H.~Ng, E.~C.~Rowell, and Z.~Wang,
\emph{Rank-finiteness for modular categories},
J. Amer. Math. Soc. \textbf{29} (2016), no.~3, 857--881.

\bibitem{CLS}
Y.~Choi, D.-C.~Lu, and Z.~Sun,
\emph{Self-duality under gauging a non-invertible symmetry},
J. High Energy Phys. 2024, no.~1, 142.



\bibitem{DMNO}
A.~Davydov, M.~M\"uger, D.~Nikshych, and V.~Ostrik,
\emph{The Witt group of non-degenerate braided fusion categories},
J. Reine Angew. Math. \textbf{677} (2013), 135--177.

\bibitem{DixonMortimer}
J.~D. Dixon and B.~Mortimer,
\emph{Permutation Groups}, Graduate Texts in Mathematics 163,
Springer, New York, 1996.


\bibitem{DLN}
C.~Dong, X.~Lin, and S.-H.~Ng,
\emph{Congruence property in conformal field theory},
Algebra Number Theory \textbf{9} (2015), no.~9, 2121--2166.

\bibitem{DS}
J.~Dong and A.~Schopieray,
\emph{Near-integral fusion}, preprint (2024), 
\href{https://arxiv.org/abs/2407.15955}{arXiv:2407.15955}.

\bibitem{EGNO}
P.~Etingof, S.~Gelaki, D.~Nikshych, and V.~Ostrik,
\emph{Tensor Categories}, Math. Surveys Monogr. 205,
American Mathematical Society, 2015.

\bibitem{EGNOcorr}
P.~Etingof, S.~Gelaki, D.~Nikshych, and V.~Ostrik,
\emph{Corrections to the book ``Tensor Categories''},
February 14, 2026, p.~5 (Example~7.12.26).

\bibitem{ENO}
P.~Etingof, D.~Nikshych, and V.~Ostrik, \emph{Weakly group-theoretical and solvable fusion categories}, Adv. Math. \textbf{226} (2011), 176--205.

\bibitem{ENOH}
P.~Etingof, D.~Nikshych, and V.~Ostrik,
\emph{Fusion categories and homotopy theory},
Quantum Topol. \textbf{1} (2010), no.~3, 209--273.

\bibitem{etingof2005fusion}
P.~Etingof, D.~Nikshych, and V.~Ostrik,
\emph{On fusion categories},
Ann. of Math. (2) \textbf{162} (2005), no.~2, 581--642.

\bibitem{GHR}
C.~Galindo, S.-M.~Hong, and E.~C.~Rowell,
\emph{Generalized and quasi-localizations of braid group representations},
Int. Math. Res. Not. IMRN 2013, no.~3, 693--731.

\bibitem{GN}
S.~Gelaki and D.~Naidu,
\emph{Some properties of group-theoretical categories},
J. Algebra \textbf{322} (2009), no.~8, 2631--2641.

\bibitem{Larson}
H.~K.~Larson,
\emph{Pseudo-unitary non-self-dual fusion categories of rank 4},
J. Algebra \textbf{415} (2014), 184--213.

\bibitem{Masuoka}
A.~Masuoka,
\emph{Semisimple Hopf algebras of dimension 6, 8},
Israel J. Math. \textbf{92} (1995), 361--373.

\bibitem{NS}
S.-H.~Ng and P.~Schauenburg,
\emph{Frobenius--Schur indicators and exponents of spherical categories},
Adv. Math. \textbf{211} (2007), no.~1, 34--71.

\bibitem{Ostrik}
V.~Ostrik,
\emph{Pivotal fusion categories of rank 3},
Mosc. Math. J. \textbf{15} (2015), no.~2, 373--396.

\bibitem{Shimizu}
K.~Shimizu,
\emph{Frobenius--Schur indicators in Tambara--Yamagami categories},
J. Algebra \textbf{332} (2011), 543--564.

\bibitem{Shimizu2017}
K.~Shimizu,
\emph{On unimodular finite tensor categories},
Int. Math. Res. Not. IMRN 2017, no.~1, 277--322.

\bibitem{Siehler}
J.~Siehler,
\emph{Braided near-group categories},
\href{https://arxiv.org/abs/math/0011037}{arXiv:math/0011037} (2000).

\bibitem{TY}
D.~Tambara and S.~Yamagami,
\emph{Tensor categories with fusion rules of self-duality for finite abelian groups},
J. Algebra \textbf{209} (1998), 692--707.

\bibitem{Vafa1988421}
C.~Vafa, \emph{Toward classification of conformal theories}, Phys. Lett. B \textbf{206}~(3)  (1988), 421 -- 426.

\end{thebibliography}
\end{document}